\documentclass[reqno,a4paper,11pt]{amsart}
\usepackage[top=32mm, bottom=32mm, left=32mm, right=32mm,height = 615pt, width = 360pt]{geometry}
\usepackage{mathrsfs,enumerate}
\usepackage{amssymb}
\usepackage[colorlinks=true]{hyperref}
\allowdisplaybreaks

\usepackage{amsmath}

\usepackage{tikz}
\usetikzlibrary{arrows, decorations.pathmorphing, backgrounds, positioning, fit, petri, automata,shapes,calc}

\newcommand{\ld}{\lambda}
\newcommand{\Ga}{\Gamma}
\newcommand{\ga}{\gamma}
\newcommand{\al}{\alpha}

\newcommand{\de}{\delta}

\newcommand{\ep}{\epsilon}
\newcommand{\Om}{\Omega}

\newcommand{\tht}{\theta}

\newcommand{\beqq}{\begin{equation*}}
\newcommand{\eeqq}{\end{equation*}}
\newcommand{\beq}{\begin{equation}}
\newcommand{\eeq}{\end{equation}}

\newtheorem{theorem}{Theorem}[section]
\newtheorem{lemma}{Lemma}[section]

\newtheorem{definition}{Definition}[section]

\newtheorem{corollary}{Corollary}[section]
\newtheorem{question}{Question}[section]

\newtheorem{remark}{Remark}[section]

\numberwithin{equation}{section}

\begin{document}
\title[A transcendental Julia set]{A generalized family of transcendental functions with one dimensional Julia sets}



\maketitle

\begin{center}
Xu Zhang\footnote{
Email: xu$\_$zhang$\_$sdu@mail.sdu.edu.cn
}\\
Department of Mathematics\\
Shandong University, Weihai, Shandong, 264209,  China\\
\end{center}

\begin{abstract}
A generalized family of transcendental (non-polynomial entire) functions is constructed, where the Hausdorff dimension and the packing dimension of the Julia sets are equal to one. Further, there exist multiply connected wandering domains, the dynamics can be completed described, and for any $s\in(0,+\infty]$, there is a function taken from this family with the order of growth $s$. Baker proved that the Hausdorff dimension of the transcendental function is no less than one in 1975, the minimum value was obtained via an elegant construction by Bishop in 2018. The order of growth is zero in Bishop's construction, the family of functions here have arbitrarily positive or even infinite order of growth.
\end{abstract}

\section{Introduction}

A holomorphic (analytic) function $f:\mathbb{C}\to\mathbb{C}$ defined on the whole complex plane is called entire. The entire functions include polynomials and transcendental (non-polynomial) functions. Examples of transcendental functions include the exponential function, the logarithm, and the trigonometric functions.
The dynamical behavior of $f$ is the study of the iteration of $f$ on the complex plane. Let $f^n$ denote the $n$-th iterate of $f$, $n\in\mathbb{N}$. The Fatou set $\mathcal{F}(f)$ of an entire function $f$ is the set where the
iterates $f^n$
form a normal family (sequences having convergent subsequences in the sense of Montel, i.e., $\widehat{\mathbb{C}}=\mathbb{C}\cup\{\infty\}$ with the spherical metric). The Julia set $\mathcal{J}(f)$ is the
complement of the Fatou set \cite{Milnor2006}.

 The escaping set of an entire function is defined by \cite{Eremenko1989}:
\beqq
I(f)=\{z\in\mathbb{C}:\ f^{n}(z)\to\infty\ \mbox{as}\ n\to\infty\}.
\eeqq
The Julia set is the boundary of the escaping set by a result of Er\"{e}menko \cite{Eremenko1989}. And, Baker proved that the multiply connected components of the Fatou set are in the escaping set \cite{Baker1976}.

In transcendental dynamics, the rates of escape for the escaping set is very useful. The fast escaping set \cite{BergweilerHinkkanen1999} is given by
\beq\label{fastescap8-11-1}
A(f)=\{z\in\mathbb{C}:\ \mbox{there is a}\ k\geq0\ \mbox{so that}\ |f^{n+k}(z)|\geq S_n\ \mbox{for all}\ n\geq0\},
\eeq
where $S_0$ is a fixed large number, and $S_{n+1}=\max_{|z|=S_n}|f(z)|$ inductively. For $|z|\leq S_0$, $S_n$ is an upper bound for $|f^n(z)|$, and the fast escaping set contains the points that almost achieve the upper bound. Rippon and Stallard verified that the closure of each Fatou component is in $A(f)$ \cite{RipponStallard2005b}.

The geometric structure of the Julia sets might illustrate fractal structure, the fractal dimension is a useful index for the description of the geometric objects, three useful definitions are Minkowskii dimension, Hausdorff dimension, and packing dimension (See Subsection \ref{preliminaryresult-2} for more details).

The study of the fractal dimension is an interesting topic in complex dynamics. McMullen obtained the Hausdorff dimension of the Julia for the polynomial $e^{2\pi i\al}z+z^2$ with the rotation number $\al$ of bounded type is strictly less than two \cite{McMullen1998}. Shishikura studied the Huasdorff dimension of the Mandelbrot set generated by quadratic polynomials \cite{Shishikura1998}. Baker proved that the Fatou set of a transcendental function has no unbounded, multiply connected components, implying that the Julia set can not be totally disconnected, contains a non-trivial continuum, and has Hausdorff dimension at least one \cite{Baker1975}.
Misiurewicz used an elegant argument to show that the Julia set of the exponential function $e^z$ is the whole plane, implying the Hausdorff dimension is two \cite{Misiurewicz1981}. McMullen studied an exponential family and a sine family, showed that the Julia set of any member of the exponential family has Hausdorff dimension two, and the Julia set of any member of the sine family has positive area \cite{McMullen1987}. Stallard constructed many transcendental functions such that the Hausdorff dimension can be any number in the interval $(1,2]$ \cite{Stallard1997, Stallard2000}. Later, Christopher Bishop provided an infinite product construction of a transcendental function with Hausdorff dimension and packing dimension one \cite{Bishop2018}, which solved an open problem of Baker since 1975.  The method of Bishop was also extended in many other examples, including Baker's original example on the existence of wandering Fatou domains \cite{Baumgartner2015}. There are many work on the study of Hausdorff dimension \cite{BaranskiKarpinskaZdunik2009, BergweilerKarpinska2010, BergweilerKarpinskaStallard2009, RempeStallard2010, Sixsmith2015}.

The singular value can be used in the classification of the dynamics of the transcendental functions. Let $f$ be an entire function and $\al\in\widehat{\mathbb{C}}=\mathbb{C}\cup\{\infty\}$. The number $\al$ is called a singular value if $f$ is not a smooth covering map over any neighborhood of $\al$. We denote the set of all singular values by $\mbox{sing}(f^{-1})$. In other words, if $\al$ is a non-singular value of $f$, then there exists a neighborhood $V$
 of $\al$, where every branch of $f^{-1}$ in $V$ is well defined and is a conformal map of $V$. In the case of a rational function $f$, $\mbox{sing}(f^{-1})$ is nothing but the set of all critical values, that is the images of critical points.
Based on the singular values, three special classes of entire functions can be defined for transcendental functions:
\beqq
\mathcal{B}=\{f:\ \mbox{sing}(f^{-1})\ \mbox{is a bounded set}\},
\eeqq
\beqq
\mathcal{S}=\{f:\ \mbox{sing}(f^{-1})\ \mbox{is a finite set}\},
\eeqq
\beqq
\mathcal{C}=\{f:\ d(S^{+}(f),\mathcal{J}(f))>0\}\ \mbox{with}\ S^{+}(f)=\cup_{n\geq0}f^n(\mbox{sing}(f^{-1})).
\eeqq
The class of functions $\mathcal{B}$ is called entire functions of bounded singular type or Er\"{e}menko-Lyubich class introduced by Er\"{e}menko and Lyubich \cite{EremenkoLyubich1987}, where the singular sets are bounded (but possibly infinite). A transcendental entire function in $\mathcal{S}$ is said to be of finite singular type or to belong to the Speiser class, where the Fatou sets of the functions in $\mathcal{S}$ does not have wandering domains or Baker domains, similar with dynamics for polynomials. The Hausdorff dimension of the Julia sets of Er\"{e}menko-Lyubich functions is strictly larger than $1$ \cite{Stallard1996}, whereas the packing dimension is always $2$ \cite{RipponStallard2005}. Hence, the examples with both Hausdorff dimension and packing dimension $1$ are not in the Er\"{e}menko-Lyubich class.

The order of growth is given by
\beqq
\rho(f)=\limsup_{z\to\infty}\frac{\log\log|f(z)|}{\log|z|}.
\eeqq
 The order of growth has important applications in dynamics via geometric function theory by Rottenfusser et al. \cite{RottRuckRempSchl2011}. A conjecture of Baker is that, ``if
an entire function has order less than $1/2$, does this imply that the Fatou set has no unbounded Fatou components?". A lot of work has contributed to this problem \cite{RipponStallard2013}.
The order of growth can be used in the study of the area of the escaping set and the Julia set for entire functions \cite{Bergweiler2017}.

In this paper, we provide the construction of transcendental functions with positive or even infinite order of growth:
\begin{theorem}\label{maindim8-13-1}
For any $s\in(0,+\infty]$, there is a transcendental entire function $f$ such that the Julia set has
finite $1$-dimensional spherical Hausdorff measure and the order of growth is $s$.
\end{theorem}
\begin{theorem}\label{maindim8-13-2}
 For any $s\in(0,+\infty]$, there is a transcendental function a function $f$ with the order of growth $s$ satisfying the following properties:
\begin{itemize}
\item[(1)] Every Fatou component $\Om$ is a bounded, infinitely connected domain whose boundary consists of a countable number of $C^1$ curves, and the accumulation set of these curves is the outer component of $\partial \Om$, where this boundary separates $\Om$ from $\infty$.
\item[(2)] The fast escaping set, $A(f)$, is the union of the closure of all the Fatou components, and $A(f)\cap\mathcal{J}(f)$ is the union of boundaries of the Fatou components.
\item[(3)] $\mathcal{J}(f)$ has Hausdorff dimension and packing dimension $1$.
\item[(4)] Given any $\al>0$, $f$ may be chosen so that $\mbox{dim}(\mathbb{C}\setminus A(f))<\al$.
\item[(5)] $\mbox{dim}(I(f)\setminus A(f))=0$.
\end{itemize}
\end{theorem}

Since the Hausdorff and packing dimension is $1$ in the work of Bishop \cite{Bishop2018}, these results are the generalization of Bishop's work.

We also provide examples to illustrate a criteria provided by Bergweiler is sharp. In \cite{Bergweiler2012}, Bergweiler obtained a criteria on the estimation of the packing dimension:
\begin{theorem}\cite[Theorem 1.1]{Bergweiler2012}
Let $f$ be a transcendental entire function satisfying
\beq\label{est-3}
\liminf_{r\to\infty}\frac{\log\log(\max_{|z|=r}|f(z)|)}{\log\log r}=\infty.
\eeq
If $\mathcal{F}(f)$ has no multiply connected component, then
\beqq
\mbox{Pdim}(I(f)\cap \mathcal{J}(f))=2.
\eeqq
\end{theorem}
In our work, we show that there is a transcendental function, which satisfies \eqref{est-3} and $\mathcal{F}(f)$ has a multiply connected component, but the packing dimension is $1$. Further, for any $s\in(0,+\infty)$, there is a function $f$ satisfying
\beqq
\liminf_{r\to\infty}\frac{\log\log(\max_{|z|=r}|f(z)|)}{\log\log r}=s,
\eeqq
the Fatou set of this function has a multiply  connected component, and the packing dimension of the Julia set in the escaping set is $1$ (See Remark \ref{criteria-Ber}).

This solves an open problem of Bishop:
\begin{question} \cite[Problem 4]{Bishop2018}
 The examples constructed in Bishop's work \cite{Bishop2018} have order of growth zero, where the order of growth is zero, moreover, the construction there can be as ``close to" polynomial growth as we wish. Can we build examples of positive or infinite order of growth?  Can we use such constructions to show the conditions in Bergweiler's paper \cite{Bergweiler2012} implying $\mbox{Pdim}(\mathcal{J})=2$ are sharp?
\end{question}
This also provides a solution to a problem of Baker since 1975 with Hausdorff and packing dimension $1$, and any order of growth. This kind of examples also provide evidence on the correctness of Baker's conjecture on the order of growth. Further, by combing the techniques used in our present work and the recent work of \cite{Burkart2019}, one could construct transcendental functions with packing dimensions dense in the interval $(1,2)$, and finite or even infinite order of growth.
And, these examples are not in the Er\"{e}menko-Lyubich class or Speiser class.

The rest of this paper is organized as follows. In Section \ref{preliminaryresult}, some useful concepts and lemmas are introduced, this section is divided into three parts. In Subsection \ref{preliminaryresult-1}, some concepts and results in complex dynamics are given; in Subsection \ref{preliminaryresult-2}, some results on fractal dimension are introduced; in Subsection \ref{hypdy-1}, a useful class of hyperbolic polynomials is introduced.
The main idea of the construction of this kind of functions is contained in Section \ref{functiondef}, The order of growth is estimated in Subsection \ref{ordergrowth1}.
The details of the construction of the transcendental functions are provided in Section \ref{funconstr}, the whole construction is divided into several steps.
The packing dimension is obtained in Section \ref{packingdim-1}.

\section{Preliminaries}\label{preliminaryresult}

In this section, some useful results are introduced. This section is split into three parts. In Subsection \ref{preliminaryresult-1}, some concepts and results in complex dynamics are given; in Subsection \ref{preliminaryresult-2}, some results on fractal dimension are introduced; in Subsection \ref{hypdy-1}, a useful class of hyperbolic polynomials is introduced.

\subsection{Basic concepts}\label{preliminaryresult-1}

\begin{definition} \cite{Milnor2006}
Let $\widehat{\mathbb{C}}=\mathbb{C}\cup\{\infty\}$ be the extended complex plane (or the one-point compactification of $\mathbb{C}$), the spherical metric $\chi$ defined on
$\widehat{\mathbb{C}}$ is given by
\beqq
\chi(z,z')=\frac{|z-z'|}{\sqrt{1+|z|^2}\sqrt{1+|z'|^2}}\ \forall z,z'\in\mathbb{C};\ \ \chi(z,\infty)=\frac{1}{\sqrt{1+|z|^2}}.
\eeqq
Let $D$ be a domain in the complex plane $\mathbb{C}$. A family $\mathcal{F}$ of meromorphic functions on $D$ is said to be normal on $D$ if each sequence $\{f_n\}\subset \mathcal{F}$ has a convergent subsequence on compact subsets of $D$ with respect to the spherical metric.
\end{definition}

\begin{definition} \cite{Milnor2006}
 The Fatou set $\mathcal{F}(f)$ of an entire function $f$ is the set where the
iterates $f^n$ locally
form a normal family. The Julia set $\mathcal{J}(f)$ is the
complement of the Fatou set.
\end{definition}

\begin{lemma} (Cauchy Formula) \label{cauchyformula8-23-1}
Let $f$ be an analytic function defined on $B(a,R)$. Assume for any $z\in B(a,R)$, $|f(z)|\leq M$.
Then, for any $n\in\mathbb{N}$, one has
\beqq
|f^{(n)}(a)|\leq\frac{n!M}{R^n}.
\eeqq
\end{lemma}

\begin{lemma}\cite[Theorem 3.17]{Morosawa2000}
For any entire function, a multiply connected component of the Fatou set is a wandering domain.
\end{lemma}

\subsection{On the definitions of dimensions}\label{preliminaryresult-2}

In this section, the concept of Minkowski, Hausdorff, packing dimensions on fractal geometry are introduced, and the Whitney decomposition is also introduced, which is a useful tool in the estimate of dimension \cite{BishopPeres2017, Falconer1990}.

\begin{definition} (Minkowski dimesnion)
Let $X$ be a metric space. For a bounded set $K\subset X$ and any positive number $\ep$, consider the family of all the subsets of $X$ with diameter no larger than $\ep$, let $N(K,\ep)$ be the minimal number of subsets from this family such that the union of these subsets covers $K$. The upper and lower Minkowski dimension are defined respectively by
\beqq
\overline{\mbox{Mdim}}(K)=\limsup_{\ep\to0}\frac{\log N(K,\ep)}{\log 1/\ep}
\eeqq
and
\beqq
\underline{\mbox{Mdim}}(K)=\liminf_{\ep\to0}\frac{\log N(K,\ep)}{\log 1/\ep}.
\eeqq
If $\overline{\mbox{Mdim}}(K)=\underline{\mbox{Mdim}}(K)$, then this is called the Minkowski dimension of $K$, denoted by $\mbox{Mdim}(K)$.
\end{definition}

\begin{definition} (Hausdorff dimension)
Let $X$ be a metric space, $K$ be a subset of $X$, and $\al$ be a positive number.
\begin{itemize}
\item the $\al$-dimensional Hausdorff content is
\beqq
\mathcal{H}^{\al}_{\infty}(K)=\inf_{\mathcal{U}}\bigg\{\sum_i \mbox{diam}(U_i)^{\al}:\ K\subset\bigcup_{i} U_i\bigg\},
\eeqq
where the infimum is with respect to all the countable cover $\mathcal{U}=\{U_i\}_{i\in\mathbb{N}}$ of $K$.
\item The Hausdorff dimension of $K$ is
\beqq
\mbox{dim}(K)=\inf\{\al:\ \mathcal{H}^{\al}_{\infty}(K)=0\}.
\eeqq
\item For any positive number $\ep$, set
\beqq
\mathcal{H}^{\al}_{\ep}(K):=\inf_{\mathcal{U}}\bigg\{\sum_i\mbox{diam}(U_i)^{\al}:\ K\subset\bigcup U_i,\ \mbox{diam}(U_i)<\ep\bigg\},
\eeqq
where the infimum is with respect to all
the countable cover $\mathcal{U}=\{U_i\}_{i\in\mathbb{N}}$ with diameter less than $\ep$.
The $\al$-dimensional Hausdorff measure of $K$ is
\beqq
\mathcal{H}^{\al}(K)=\lim_{\ep\to0}\mathcal{H}^{\al}_{\ep}(K).
\eeqq
\end{itemize}
\end{definition}

\begin{definition} (Packing dimension)
Let $X$ be a metric space, $K$ be a subset of $X$, $\al$ be a positive number, and $\ep$ be a positive number. For any positive number $r$, let $B(x,r)=\{y\in X:\ y\in X, \mbox{dist}(x,y)<r\}$. Let $\{B(x_j,r_j)\}_{j\in\mathbb{N}}$ be a collection of disjoint open balls with center contained in $K$ and radius less than $\ep$, that is, $x_j\in K$ and $r_j<\ep$ for any $j\in\mathbb{N}$, and $\{B(x_j,r_j)\}_{j\in\mathbb{N}}$ be a cover of $K$.
\begin{itemize}
\item The $\al$-dimensional packing premeasure is
\beqq
\widetilde{\mathcal{P}}^{\al}(E)=\lim_{\ep\to0}\bigg(\sup\sum^{\infty}_{j=1}(2r_j)^{\al}\bigg),
\eeqq
where the supremum is taken over all the collection of disjoint open balls.
\item The packing measure in dimension $\al$ is
\beqq
\mathcal{P}^{\al}(K)=\inf\bigg\{\sum^{\infty}_{i=1}\widetilde{\mathcal{P}}^{\al}(K_i):\ K\subset\bigcup^{\infty}_{i=1}K_i\bigg\}.
\eeqq
\item The packing dimension of $K$ is
\beqq
\mbox{Pdim}(K)=\inf\{\al:\ \mathcal{P}^{\al}(K)=0\}.
\eeqq
\end{itemize}
\end{definition}

\begin{lemma}\cite[Proposition 2.7]{BishopPeres2017}\label{equest-2}
The packing dimension of any set $K$ in a metric space may be expressed in terms of upper Minkowski dimensions:
$$
\mbox{Pdim}(K)=\inf\bigg\{\sup_{j\geq1}\overline{\mbox{Mdim}}(K_j):\ K\subset\cup^{\infty}_{j=1}K_j\bigg\},
$$
where the infimum is over all countable covers of $K$. Since the upper Minkowski dimension of a set and its closure are the same, we can assume that all the sets $\{K_j\}$ above are closed.
\end{lemma}

\begin{lemma}\label{equest21-1}
By the definitions above, one has
\beqq
\mbox{dim}(K)\leq\underline{\mbox{Mdim}}(K)\leq\overline{\mbox{Mdim}}(K)
\eeqq
and
$$
\mbox{dim}(K)\leq \mbox{Pdim}(K)\leq\overline{Mdim}(K).
$$
\end{lemma}

Now, the Whitney decomposition is introduced \cite{Stein1970}. The dyadic cubes and Whitney covers can be used in the definition of the upper Minkowski dimension. For $n\in\mathbb{Z}$, the collection of $n$-th generation of closed dyadic intervals $Q=[j2^{-n},(j+1)2^{-n}]$, the length is $l(Q)=2^{-n}$. Denote $\mathcal{D}=\cup_{n\in\mathbb{Z}}\mathcal{D}_n$. A dyadic cube in $\mathbb{R}^d$ is any product of dyadic intervals that all have the same length, the length of a square is $l(Q)$ and the diameter is $|Q|=\sqrt{d}|Q|$. Each dyadic cube is contained in a unique dyadic cube $Q^{\uparrow}$ with $|Q^{\uparrow}|=2|Q|$, this $|Q^{\uparrow}|$ is called the parent of $Q$.

Let $\Om\subset\mathbb{R}^d$ be an open subset. Every point of $\Om$ is contained in a dyadic cube $Q$ with $Q\subset\Om$ and $|Q|\leq\mbox{dist}(Q,\partial\Om)$. By maximality, there is a collection of dyadic cubes satisfying $\mbox{dist}(Q^{\uparrow},\partial\Om)\leq|Q^{\uparrow}|$, implying that $\mbox{dist}(Q,\partial\Om)\leq|Q^{\uparrow}|+|Q|=3|Q|$. This collection of dyadic cubes is called a Whitney decomposition, that is, a collection of dyadic cubes $\{Q_j\}$ in $\Om$ disjoint except along their boundaries, whose union covers $\Om$ and
\beqq
\frac{1}{\ld}\mbox{dist}(Q_j,\partial\Om)\leq|Q_j|\leq\ld\mbox{dist}(Q_j,\partial\Om)
\eeqq
for some constant $\ld>1$ (see Theorem 3 in \cite{Stein1970}).

\begin{definition}
For any compact set $K\subset\mathbb{R}^d$, a Whitney decomposition $\mathcal{W}$
is for $\Om=\mathbb{R}^d\setminus K$ that are within distance $1$ of $K$, the exponent
of convergence is defined by
\beqq
\al=\al(K)=\inf\bigg\{\al:\ \sum_{Q\in\mathcal{W}}|Q|^{\al}<\infty\bigg\}.
\eeqq
\end{definition}

\begin{lemma}\cite[Lemma 2.6.1]{BishopPeres2017}\label{equest-3}
For any compact set $K\subset\mathbb{R}^d$, one has $\al(K)\leq\overline{\mbox{Mdim}}(K)$. If the Lebesgue measure of $K$ is zero, then $\al(K)=\overline{\mbox{Mdim}}(K)$.
\end{lemma}

The Minkowski dimension of the system constructed here is verified to be $1$ by the Whitney decomposition. This, together with Lemmas \ref{equest-2}, \ref{equest21-1}, and \ref{equest-3}, implies the packing and Hausdorff dimension is $1$.

\subsection{Hyperbolic dynamics of polynomial maps}\label{hypdy-1}

In this section, a polynomial map with hyperbolic dynamics and related properties are introduced \cite{Morosawa2000}.

For a polynomial
\beqq
 p(z)=a_{k_0}z^{k_0}+a_{k_0-1}z^{k_0-1}+\cdots+a_0,
\eeqq
the set of escaping points is
\beqq
I(p)=\{z\in\mathbb{C}:\ \lim_{n\to\infty}p^n(z)=\infty\};
\eeqq
the set of critical points is
\beqq
C(p)=\{z\in\mathbb{C}:\ p^{\prime}(z)=0\};
\eeqq
the post-critical set is
\beqq
\overline{C^{+}(p)}=\overline{\bigcup^{\infty}_{n=1}p^n(C(p))};
\eeqq
the set $\mathcal{K}(p)=\mathbb{C}\setminus I(p)$ is called the filled Julia set, the boundary of $\mathcal{K}(p)$ is said to be the Julia set, denoted by $\mathcal{J}(p)$. The polynomial $p$ with degree no less than $2$ is said to be hyperbolic, if $\mathcal{J}(p)\cap\overline{C^{+}(p)}=\emptyset$.

A simple model for the construction of the transcendental function is the polynomial $p_{\ld}(z)=\ld(2z^2-1)$ with the real parameter $\ld\geq1$. For $\ld=1$, $1$ is a fixed point of $p_{\ld}(z)$. For $\ld>1$, the orbit of $0$ is divergent to $\infty$. So, for any positive integer $n$ and $\ld\geq1$, $|p^n_{\ld}(0)|\geq1$.

\begin{lemma}\cite[Lemma 4.1]{Bishop2018}\label{equ2021-2-13-1}
For the polynomial $p_{\ld}(z)=\ld(2z^2-1)$ with $\ld\geq1$, the Julia set is a Cantor subset of $[-1,1]$, and
the upper Minkowski dimension tends to zero as $\ld\to\infty$.
\end{lemma}

\begin{proof}
An outline of the arguments is provided. A conjugacy map $\tfrac{1}{2}(z+\tfrac{1}{z})$ conjugates the action of $z^2$ on $\mathbb{D}=\{z:\ |z|>1\}$ to the action of $T_2(z)=2z^2-1$ on $U=\mathbb{C}\setminus[-1,1]$. So, the Julia set for $T_2$ is contained in $[-1,1]$, and the iteration of points off $[-1,1]$ escapes to $\infty$.  For $\ld>1$, the Julia set of $p_{\ld}$ is a Cantor set contained in the following two intervals depending on $\ld$:
$$\bigg[-\sqrt{\frac{1}{2}+\frac{1}{2\ld}},-\sqrt{\frac{1}{2}-\frac{1}{2\ld}}\bigg]\cup
\bigg[\sqrt{\frac{1}{2}-\frac{1}{2\ld}},\sqrt{\frac{1}{2}+\frac{1}{2\ld}}\bigg].$$
\end{proof}

\begin{lemma}\cite[Lemma 4.2]{Bishop2018}\label{equ2021-2-13-2}
Consider the polynomial $p_{\ld}(z)=\ld(2z^2-1)$ with $\ld\geq1$. For any $r\geq2$ and $n\in\mathbb{N}$,
let $\ga_n$ be a connected component of $\{z:\ |p^n_{\ld}(z)|=r\}$. There is a constant
$C_{\ld}$ such that $\text{diam}(p_{\ld}(\ga_n))\geq C_{\ld}\text{diam}(\ga_n)$ and $C_{\ld}$ may
be chosen as large as we wish by taking $\ld$ large enough.
\end{lemma}

\section{The outline of the construction of the function}\label{functiondef}

The main idea of the construction of the function is provided in this section.

Recall that an annulus is the bounded area between two concentric circles, and the width of an annulus is the difference between the radii of its outer and inner bounding circles.

Consider a function
\beq\label{equ2021-7-25-1}
F_0(z)=p_{\ld}^N(z),
\eeq
where $F_0$ is the $N$-th iterates of the polynomial $p_{\ld}(z)$.

\begin{remark}\label{equ2021-7-29-1}
For the polynomial $F_0$, the critical values are the first iteration of the critical points. By \eqref{equ2021-7-25-1}, $F_0$ has $2^N-1$ critical points. Since $1+2+2^2+\cdots+\cdots+2^{N-2}+2^{N-1}=2^N-1$, the critical points consist of this set
$\{0,\ p^{-1}_{\ld}(0),\ p^{-2}_{\ld}(0),...,p^{-{(N-1)}}_{\ld}(0)\}$, that is, the critical point $0$, the two pre-images of $0$ under $p_{\ld}$, and so on. Hence, the critical values of $F_0$ are the first $N$ iterates of $0$ under $p_{\ld}$. Further, it follows from \eqref{equ2021-7-25-1} that the $F_0$-images of the critical values of $F_0$ are from the $(N+1)th$ to $(2N)th$ iterates of $0$.
\end{remark}

Let $m=2^N$, by \eqref{equ2021-7-25-1}, $F_0$ is a polynomial with degree $m$, and has the leading term (the highest degree term) is $(2\ld)^{m-1}z^m$. So,  there is a sufficiently large positive number $R\geq32$ such that
\beq\label{critical-1}
\frac{1}{2}\leq\bigg|\frac{F_0(z)}{(2\ld)^{m-1}z^{m}}\bigg|\leq\frac{3}{2}\ \mbox{for}\ |z|\geq R.
\eeq
For convenience, take a positive constant $m^*=\tfrac{(m-1)\log(2\ld)}{\log(\ld)}$ such that
\beqq
(2\ld)^{m-1}=\ld^{m^*}.
\eeqq

Assume $\ld$ is sufficiently large such that the dimension of the Cantor set near the origin is sufficiently small by Lemma \ref{equ2021-2-13-1}, assume $R$ is sufficiently large such that the dimension of the Cantor set for the perturbation map of $F_0$ is also sufficiently small.

Fix a positive constant $L_0>1$.  Take $R$ large enough,  set
\beq
R_1:=2R,
\eeq
choose a positive integer $n_1$, and define
\beq
F_1(z)=1-\frac{1}{2}\bigg(\frac{z}{R_1}\bigg)^{n_1}.
\eeq
Define
\beqq
f_0(z)=F_0(z)\ \text{and}\ f_1(z)=F_0(z)F_1(z).
\eeqq
Inductively, define
\beq
f_k(z)=f_{k-1}(z)\cdot F_k(z)=\prod^{k}_{j=0}F_j(z).
\eeq
that is, suppose the polynomial $F_k$ has been defined, denote
\beq\label{inequ-12a}
R_{k+1}:=M(f_k,2R_k)=\max\{|f_k(z)|:\ |z|=2R_k\},
\eeq
choose a positive integer $n_{k+1}$, and set
\beq
F_{k+1}(z):=1-\frac{1}{2}\bigg(\frac{z}{R_{k+1}}\bigg)^{n_{k+1}}
\eeq
and
\beqq
f_{k+1}(z):=f_{k}(z)\cdot F_{k+1}(z)=\prod^{k+1}_{j=0}F_j(z).
\eeqq
Define the function $f$ as
\beq\label{equ2021-1-26-3a}
f(z)=\lim_{k\to\infty}f_{k}(z)=\prod^{\infty}_{k=0}F_k(z)=F_0(z)\cdot\bigg[\prod^{\infty}_{k=1}F_k(z)\bigg].
\eeq
For the above chosen sequence of positive integers $\{n_{k}\}_{k\in\mathbb{N}}$, if
\begin{itemize}
\item[(A1)]  $\sum^{\infty}_{k=1}\tfrac{1}{(L_0)^{n_k}}<\infty,$
\end{itemize}
then this function $f$ is well-defined by Lemma \ref{convergrad}.

\begin{remark}
In \cite{Bishop2018}, the parameters $n_{k}$ are preassigned numbers depending on a function controlling the order of growth, such that the order of growth is zero, and the growth is as close to polynomial growth as we wish.

The freedom of the choices of the parameters $n_{k}$ is used to change the order of the growth, implying that the order can be positive or even infinite. So, from this point of view, the construction of this article can be thought of as a complementary of the work in \cite{Bishop2018}, where the order of growth is zero there. If $n_k=(\lfloor R_k\rfloor)^k$, then the order is $+\infty$; if  $n_k=\lfloor R_k^s\rfloor$, then the order is $s\in(0,+\infty)$ (For more details, see Subsection \ref{ordergrowth1}).
\end{remark}

In Assumptions (A1),  the convergence of $\sum^{\infty}_{k=1}\tfrac{1}{L_0^{n_k}}$ requires that $n_k$ should grow to infinity as $k$ goes to $+\infty$.
For clarity of the discussions, assume $L_0=\tfrac{3}{2}$, a simpler assumption instead of (A1) is introduced:
\begin{itemize}
\item[(A*)] $n_0=7$, $n_1\geq n_0+1$, and $n_{i+1}>n_{i}$ for any $i\geq1$.
\end{itemize}

The constant $R$ should be large enough, see Lemmas \ref{seriescon-8-4-2}, \ref{boundaryfatou}, \ref{inequ-15}, \ref{criticalfatou-8-17-3}, and \ref{equ-2021-2-13-3}.
For convenience, introduce the following assumption:
\begin{itemize}
\item[(A**)] $m\geq 2^4=16$, $R\geq 2^5$, where $R$ should be large enough.
\end{itemize}

\begin{remark}
The derivation of some useful properties of the function $f$ needs the assumption $n_{k}\geq8$. So, we assume $n_0=7$.
\end{remark}

By direct calculation, one has
\beq\label{est-8}
m_k:=\mbox{deg}(f_k)=\sum^{k}_{j=0}\mbox{deg}(F_j)=2^N+\sum^{k}_{j=1}n_j=m+\sum^{k}_{j=1}n_j,\ k\geq1.
\eeq

The zeros of $F_k$ are evenly spaced near a circle of radius
\beq\label{valueest-18}
r_k=R_k\bigg(1+\frac{\log2}{n_k}+O(n^{-2}_k)\bigg).
\eeq

\begin{lemma}\cite[Lemma 4.3]{Bishop2018}\label{lemm2021725-1}
For any positive integer $n$, denote by $p^n_{\ld}(z)$ the $n$-th iterate of $p_{\ld}$. Then $|(p^{n}_{\ld})^{\prime\prime}(0)|\geq(4\ld)^n$.
\end{lemma}

\begin{lemma}
Suppose $\ld\geq1$, one has $\lim_{k\to\infty}R_k\to+\infty$ and
\beq\label{inequ-1}
R_{k+1}\geq 4 R_k^2.
\eeq
\end{lemma}

\begin{proof}
It follows from the product rule of derivatives, $F_k(0)=1$,  $F^{\prime}_k(0)=F^{\prime\prime}_k(0)=0$ for $k\geq1$, and Lemma \ref{lemm2021725-1}, that
\beqq
f^{\prime\prime}_k(0)=\sum^{k}_{j=0}F^{\prime\prime}_j(0)\prod^{k}_{\substack{l=0\\
l\neq j}}F_l(0)+\sum^{k}_{j=0}\sum^k_{\substack{n=0\\ n\neq j}}\bigg(F^{\prime}_j(0)F^{\prime}_n(0)\prod^k_{\substack{l=0\\ l\neq j,n}}F_l(0)\bigg)=F^{\prime\prime}_0(0).
\eeqq
So, $|f^{\prime\prime}_k(0)|=|F^{\prime\prime}_0(0)|\geq(4\ld)^N$.

By the Cauchy formula, one has
\beqq
\ld\leq|f^{\prime\prime}_k(0)|\leq\frac{2M(f_k,r)}{r^2}\ \mbox{for}\ r>0.
\eeqq
Let $r=2R_k$, one has
\beqq
R_{k+1}\geq\frac{1}{2} (2R_k)^24\ld\geq8 R_k^2\ld> 4 R_k^2.
\eeqq
\end{proof}

\begin{remark}
The notation ``big O" will be used, where $a_k=O(b_k)$ means that there is a constant $C>0$ such that $a_k\leq C b_k$ for all $k\geq1$.
\end{remark}

\begin{lemma}\label{convergrad}
If the infinite product $f(z)$ in \eqref{equ2021-1-26-3a} satisfies (A1),
then the infinite product $f(z)$
converges uniformly on any compact subset of $\mathbb{C}$.
\end{lemma}

\begin{proof}
Given any $s>0$, take the minimal positive integer $j$ such that $R_j>L_0s$ by \eqref{inequ-1}. For $|z|\leq s$ and $k\geq j$, one has
\begin{align*}
|F_k(z)|=&\bigg|\bigg(1-\frac{1}{2}\bigg(\frac{z}{R_k}\bigg)^{n_k}\bigg)\bigg|\leq
\exp\bigg(\log\bigg(1-\frac{1}{2}\bigg|\frac{z}{R_k}\bigg|^{n_k}\bigg)\bigg)\\
\leq&
\exp\bigg( O\bigg(  \frac{1}{2}\bigg(\bigg|\frac{z}{R_k}\bigg|^{n_k}\bigg)\bigg)\bigg)
\leq
 \exp \bigg(O\bigg( \frac{1}{L_0^{n_k}}\bigg)\bigg).
\end{align*}
So,
\beqq
\bigg|\prod^{\infty}_{k=j}F_k(z)\bigg|\leq \prod^{\infty}_{k=j}|F_k(z)|\leq  \exp \bigg(O\bigg( \sum^{\infty}_{k=j}\frac{1}{L_0^{n_k}}\bigg)\bigg)<\infty.
\eeqq
Hence, the infinite product converges uniformly on the compact set $\{z:\ |z|\leq s\}$ for any $s>0$.

\end{proof}

Therefore,
\beqq
f(z)=\prod^{\infty}_{k=0}F_k(z)=F_0(z)\cdot\bigg[\prod^{\infty}_{k=1}F_k(z)\bigg]=\lim_{k\to\infty}f_{k}(z)
\eeqq
defines an entire function on the complex plane.

Set
\beq\label{equ2021-1-26-1}
A_k:=\bigg\{z:\ \frac{1}{4} R_k\leq|z|\leq 4 R_k\bigg\},\ B_k:=\bigg\{z:\ 4R_k\leq|z|\leq \frac{1}{4}R_{k+1}\bigg\},
\eeq
and
\beq
D_k:=\bigg\{z:\ |z|<\frac{1}{4} R_k\bigg\}.
\eeq

\begin{figure}
\begin{center}
\begin{tikzpicture}[scale=0.9]

\fill[lightgray] (0.5, 0) -- (1, 0) arc (0:180:1) -- (-1, 0) -- (-0.5, 0) arc (180:0:0.5);

  \fill[lightgray] (0.5, 0) -- (1, 0) arc (360:180:1) -- (-1, 0) -- (-0.5, 0) arc (180:360:0.5);

\fill[lightgray] (2, 0) -- (2.5, 0) arc (0:180:2.5) -- (-2.5, 0) -- (-2, 0) arc (180:0:2);

  \fill[lightgray] (2, 0) -- (2.5, 0) arc (360:180:2.5) -- (-2.5, 0) -- (-2, 0) arc (180:360:2);

  \fill[lightgray] (4, 0) -- (4.5, 0) arc (0:180:4.5) -- (-4, 0) -- (-4, 0) arc (180:0:4);

  \fill[lightgray] (4, 0) -- (4.5, 0) arc (360:180:4.5) -- (-4, 0) -- (-4, 0) arc (180:360:4);

  \draw (0, 0) circle (0.5);
  \draw (0, 0) circle (1);
  \draw (0, 0) circle (2);
  \draw (0, 0) circle (2.5);

  \draw (0, 0) circle (4);
  \draw (0, 0) circle (4.5);
  \draw (0, 0) circle (7);
\node at (0,0) {$D_k$};
\node at  (0.72,0) {$A_k$};
\node at  (1.5,0) {$B_k$};
\node at  (2.3,0) {\tiny{$A_{k+1}$}};
\node at  (3.2,0) {$B_{k+1}$};
\node at  (4.3,0) {\tiny{$A_{k+2}$}};
\node at  (5.5,0) {$B_{k+2}$};
\end{tikzpicture}
\end{center}
\caption{An illustration diagram of $D_k$, $A_k$, and $B_k$, where each $A_k$ has bounded modulus, the moduli of $B_k$ become bigger and bigger}
\end{figure}
By the definitions of $A_k$, $B_k$, and $D_k$, each $A_k$ has bounded modulus, the moduli of $B_k$ become bigger and bigger because of \eqref{inequ-1}, and $D_k$ is the bounded complementary component of $A_k$. We will show that (see Subsection \ref{inclusion8-26-1})
\beq\label{equ2021-1-26-2}
f(B_k)\subset B_{k+1}\ \text{and}\ A_{k+1}\subset f(A_k)\subset D_{k+2},\ k\geq1.
\eeq

It follows from $f(B_k)\subset B_{k+1}$ that the points in $B_k$ iterate locally uniformly to $\infty$. This, together with the fact that the Julia set is the boundary of the escaping set by a result of Er\"{e}menko \cite{Eremenko1989}, yields that all $B_k$ are in the Fatou set. So, the Julia set is contained in $D_1\bigcup\cup_{k\geq1} A_k$.

On $D_1$, $f$ could be thought of as a small perturbation of $F_0$. This is derived by $F_k(z)\approx 1$ for $z\in D_1$ and $k\geq1$. Since $F_0$ is an iterate of $p_{\ld}(z)$, $p_{\ld}(z)=\ld p(z)$ is hyperbolic, and the Julia set of $p(z)$ is a Cantor set, one has that the Julia set of $F_0$ is a Cantor set with small dimension for sufficiently large $\ld>0$. So, $f$ also has an invariant Cantor set with small dimension near the origin, denoted by $E$. So, the points in $D_1\setminus E$ will escape out of $D_1$ under forward iteration.

For the Julia set $\mathcal{J}(f)$ for $f$, there are some points that iterate into $E$, denoted by $\widetilde{E}$. Set
\beqq
X:=\mathcal{J}(f)\setminus\widetilde{E},
\eeqq
where this set consists of points whose orbits are in $\cup_{k\in\mathbb{N}}A_k$ infinitely many often.

Points that are mapped under $f$ into $\mathcal{J}(f)\cap(D_1\setminus E)$ eventually re-enter $A_1$, that is, they are in some pre-images of $A_1$. So, we could define the non-negative indices as follows:
\beq\label{equ2021-1-26-6}
A_{-k}=f^{-k-1}(A_1)\cap D_1,\ k\geq0.
\eeq
For the point $z\in X$, the orbit of $z$ is contained in the set $A=\bigcup_{k\in\mathbb{Z}}A_k$, a sequence of integers $k(z,n)$ can be defined such that $f^n(z)\in A_{k(z,n)}$,
\beq\label{equ2021-1-26-4}
k(z,n+1)\leq k(z,n)+1\ \mbox{for}\ k(z,n)\geq1
\eeq
and
\beq\label{equ2021-1-26-5}
k(z,n+1)= k(z,n)+1\ \mbox{for}\ k(z,n)\leq0,
\eeq
where \eqref{equ2021-1-26-4} is derived by \eqref{equ2021-1-26-2} and $f(A_k)\cap A_j=\emptyset$ for $j>k+1\geq2$ (see Lemma \ref{fatou8-2-1}), and \eqref{equ2021-1-26-5} is derived by \eqref{equ2021-1-26-6}.

Hence, the integer sequences $\{k(z,n)\}$ can be classified into two parts: the sequence is eventually strictly increasing or is not, denoted by $Z$ and $Y$, respectively:
\begin{itemize}
\item[(1)] $k(z,n+1)\leq k(z,n)$ infinitely often, denoted by $Y$, corresponding to small dimension;
\item[(2)] $k(z,n+1)=k(z,n)+1$ for all sufficiently large $n$, denoted by $Z$, corresponding to dimension $1$.
\end{itemize}
By definitions of $Y$ and $Z$, $Y$ and $Z$ are subsets of $X$. The set $Y$ contains points that do not escape very quickly, that is, the orbits with initial points in $Y$ might escape slowly, remain bounded, or oscillate. And, the dimension of $Y$ can be as small as possible. The set $Z$ is contained in the fast escaping part of the Julia set, i.e., $Z=\mathcal{J}(f)\cap A(f)$, which consist of the boundary of the components of the Fatou set, and is closed $C^1$ curves, where $\mathcal{J}(f)$ is the Julia set and $A(f)$ is the fast escaping set in \eqref{fastescap8-11-1}.  The dimensions of these two subsets will be studied in Lemmas \ref{packone} and \ref{smalldim-1}.

An illustration diagram of a connected component of the Fatou set is provided in Figure \ref{fatoucomp}.
The outer boundary of this component is smooth in the sense of $C^1$, this smooth curve separates the
this component from $\infty$ and is the accumulation set of other boundary curves, which are
grouped according to levels of curves which are roughly  cocentric with the outer boundary. The levels of
boundary curves lie in some annulus $A_k$ of bounded modulus, but the component contains the
annulus $B_{k-1}$ of huge modulus.

\begin{figure}
\begin{center}
\begin{tikzpicture}[scale=0.5]
\fill[lightgray] (0:0) -- (0:8) arc (0:360:8) -- cycle;
\fill[white] (0, 0) -- (0:1) arc (0:360:1) -- cycle;

\fill[white] (0, 4) -- (0:5) arc (0:360:0.5) -- cycle;
\fill[white] (0, -3) -- (0:-4) arc (0:360:0.5) -- cycle;

\fill[white]  (90:4.5)--(90:5) arc (90:450:0.5) -- cycle;
\fill[white]  (270:4.5)--(270:5) arc (-90:270:0.5) -- cycle;

\foreach \j in {1,2,3,...,12}
{
  \fill[white] (30*\j:6.3) circle (0.2);
};

\foreach \j in {1,2,3,...,30}
{
  \fill (12*\j:7.5) circle (2pt);
};
\end{tikzpicture}
\end{center}
\caption{An illustration diagram of a connected component of the Fatou set (adopted from Figure 1 in \cite{Bishop2018}). The outer boundary curve is $C^1$, which is the accumulation set of other boundary curves; these curves are grouped into levels which lie on curves roughly parallel to the outer boundary. This component contains an annulus $B_{k-1}$ with large modulus, and the outer boundary curve is contained in an annulus $A_k$ with bounded modulus.}\label{fatoucomp}
\end{figure}

\section{The construction of the function}\label{funconstr}

In this section, the construction of entire functions is provided. This section is divided into several steps.

\subsection{Product estimates}

In this subsection, two useful product estimates are given.

\begin{lemma}\label{est-7}
Suppose the assumption (A*) holds, for $R_k$ given as above, one has
\beq\label{valueest-1}
\bigg|\prod^{\infty}_{j=k+1}F_j(z)\bigg|=1+O(R^{-1}_k)\ \mbox{for}\ |z|\leq 4R_k.
\eeq
\end{lemma}

\begin{proof}
 Now, we show the inequality by induction:
\beq\label{inequ-8}
R_j\geq 4^{1+2+2^2+\cdots+2^{j-k-1}}R^{2^{j-k}}_k=4^{(2^{j-k})-1}R^{2^{j-k}}_k\ \mbox{for}\ j\geq k+1.
\eeq
The case $j=k+1$ is derived by \eqref{inequ-1}.  Suppose the above inequality holds for $j=l$. For $j=l+1$, by \eqref{inequ-1}, one has
\begin{align*}
&R_{l+1}\geq 4R_l^2\geq 4\cdot (4^{1+2+2^2+\cdots+2^{l-k-1}}R^{2^{l-k}}_k)^2\\
=&4^{1+2+2^2+\cdots+2^{(l+1)-k-1}}R^{2^{(l+1)-k}}_k=4^{(2^{(l+1)-k})-1}R^{2^{(l+1)-k}}_k.
\end{align*}
So, one has
\begin{align*}
&\bigg|\prod^{\infty}_{j=k+1}F_j(z)\bigg|
=\bigg|\prod^{\infty}_{j=k+1}\bigg(1-\frac{1}{2}\bigg(\frac{z}{R_j}\bigg)^{n_j}\bigg)\bigg|
=\bigg|\exp\bigg(\sum^{\infty}_{j=k+1}\log\bigg(1-\frac{1}{2}\bigg(\frac{z}{R_j}\bigg)^{n_j}\bigg)\bigg)\bigg|\\
=&\bigg|\exp\bigg(\sum^{\infty}_{j=k+1}\log\bigg(1-\frac{1}{2}\bigg(\frac{z}{R_j}\bigg)^{n_j}\bigg)\bigg)\bigg|
\leq\exp\bigg(\sum^{\infty}_{j=k+1}\log\bigg(1+\frac{1}{2}\bigg|\frac{z}{R_j}\bigg|^{n_j}\bigg)\bigg)\\
\leq&\exp\bigg(\sum^{\infty}_{j=k+1}\frac{1}{2}\bigg(\frac{4R_k}{R_j}\bigg)^{n_j}\bigg)
\leq\exp\bigg(\sum^{\infty}_{j=k+1}\frac{4R_k}{2}\bigg(\frac{4R_k}{R_j}\bigg)^{n_j-1}\bigg)\\
=&\exp\bigg(\sum^{\infty}_{j=k+1}2R_k\bigg(\frac{4R_k}{R_j}\bigg)^{n_j-1}\bigg)\leq
\exp\bigg(\sum^{\infty}_{j=k+1}2R_k\bigg(\frac{4R_k}{R_j}\bigg)^{2}\bigg)\\
\leq&\exp\bigg(\sum^{\infty}_{j=k+1}2R_k\bigg(\frac{1}{4^{2^{j-k}-2}R^{2^{j-k}-1}_k}\bigg)^{2}\bigg)
\leq\exp\bigg(\frac{4}{R_k}\bigg)\leq 1+\frac{8}{R_k},
\end{align*}
where $e^x\leq 1+2x$ for $0\leq x\leq 1$ is used in the last step.
\end{proof}

\begin{lemma}\label{est-8-22-1}
Suppose the assumption (A*) holds, for $R_k$ given as above, one has
\beq\label{valueest-2}
\prod^{k-1}_{j=1}\bigg(1+\bigg(\frac{R_j}{R_k}\bigg)^{n_j}\bigg)=1+O(R_k^{-n_0/2}).
\eeq
\end{lemma}

\begin{proof}
Direct calculation gives us that
\begin{align*}
&\prod^{k-1}_{j=1}\bigg(1+\bigg(\frac{R_j}{R_k}\bigg)^{n_j}\bigg)
=\exp\bigg(\log\prod^{k-1}_{j=1}\bigg(1+\bigg(\frac{R_j}{R_k}\bigg)^{n_j}\bigg)\bigg)\\
=&\exp\bigg(\sum^{k-1}_{j=1}\log\bigg(1+\bigg(\frac{R_j}{R_k}\bigg)^{n_j}\bigg)\bigg)
\leq\exp\bigg(\sum^{k-1}_{j=1}\bigg(\frac{R_j}{R_k}\bigg)^{n_j}\bigg)\\
\leq&\exp\bigg(\sum^{k-1}_{j=1}\bigg(\frac{R_j^2}{R_k}\bigg)\bigg(\frac{R_j}{R_k}\bigg)^{n_j-1}\bigg)
\leq\exp\bigg(\sum^{k-1}_{j=1}\bigg(\frac{R_j}{R_k}\bigg)^{n_j-1}\bigg)\\
\leq& \exp\bigg(\bigg(\frac{1}{2\sqrt{R_k}}\bigg)^{n_0}\bigg(1+\frac{1}{2}+\frac{1}{4}\cdots\bigg)^{n_0}\bigg)
\leq1+2R_k^{-n_0/2},
\end{align*}
where $R_{k-1}\leq\sqrt{R_k}/2$, $R_j\leq R_{j+1}/2$ for $1\leq j\leq k-2$ by \eqref{inequ-1}, and $e^x\leq1+2x$ for $0\leq x\leq1$ are used.
\end{proof}

\subsection{The growth of $\{R_k\}$}

In this subsection, the growth of the $\{R_k\}$ is obtained by direct computation.

The classical triangle inequality gives
\beq\label{inequ-2}
\bigg|\bigg(\frac{1}{2}\bigg(\frac{|z|}{R_k}\bigg)^{n_k}-1\bigg)\bigg|\leq|F_k(z)|\leq\bigg|\bigg(\frac{1}{2}\bigg(\frac{|z|}{R_k}\bigg)^{n_k}+1\bigg)\bigg|.
\eeq

\begin{lemma}
Suppose the assumption (A*) holds, for the above $\{R_k\}$ with $k\geq1$, one has
\beq\label{inequ-3}
R_{k+1}\geq\ld^{m^{*}}\cdot2^{(\sum^{k-1}_{j=1}(2n_j-2))+(n_k-2)+m-1}\cdot R_k^{(\sum^{k-1}_{j=1}n_j)/2+m},
\eeq
\beq\label{valueest-8}
R_{k+1}\geq\ld^{m^{*}}\cdot 2^{m-1+\sum^{k}_{j=1}(n_{j}-2)} \cdot  R_{k}^{m+\sum^{k-1}_{j=1}n_{j}}\cdot\bigg[\prod^{k-1}_{j=1}{R_j}^{-n_j}\bigg],
\eeq
and
\beq\label{valueest-11}
 R_{k+1}\leq  \frac{3}{2}\ld^{m^{*}}\cdot(2R_k)^{m+\sum^k_{j=1}n_j}\cdot\bigg[\prod^{k}_{j=1} R^{-n_j}_j\bigg].
\eeq
\end{lemma}

\begin{proof}
By \eqref{critical-1}, \eqref{inequ-1}, \eqref{inequ-2}, one has $\sqrt{R_{k}}\geq 2R_j$, $\tfrac{R_k}{R_j}\geq2\sqrt{R_k}$, $k> j\geq1$, and
\begin{align*}
 R_{k+1}&=\max_{|z|=2R_k}|f_k(z)|\nonumber\\
&\geq\max_{|z|=2R_k}|F_0(z)|\cdot \bigg[\prod^{k}_{j=1}\min_{|z|=2R_k}|F_j(z)|\bigg]\nonumber\\
&\geq \frac{1}{2}\ld^{m^{*}}\cdot(2R_k)^{m} \cdot \bigg[\prod^{k}_{j=1}\bigg(\frac{1}{2}\bigg(\frac{2R_k}{R_j}\bigg)^{n_j}-1\bigg)\bigg]\nonumber\\
&\geq  \frac{1}{2}\ld^{m^{*}}\cdot(2R_k)^{m}\cdot(2^{n_k-1}-1)\cdot\bigg[\prod^{k-1}_{j=1}\bigg(\frac{1}{2}\bigg(\frac{2R_k}{R_j}\bigg)^{n_j}-1\bigg)\bigg]\nonumber\\
&\geq\frac{1}{2}\ld^{m^{*}}\cdot(2R_k)^{m}\cdot2^{(n_k-2)}\cdot \bigg[\prod^{k-1}_{j=1}\bigg(2^{2n_j-1}R_k^{n_j/2}-1\bigg)\bigg]\nonumber\\
&\geq\frac{1}{2}\ld^{m^{*}}\cdot(2R_k)^{m}\cdot2^{(n_k-2)}\cdot \bigg[\prod^{k-1}_{j=1}(2^{2n_j-2}R_k^{n_j/2})\bigg]\nonumber\\
&=\ld^{m^{*}}\cdot2^{(\sum^{k-1}_{j=1}(2n_j-2))+(n_k-2)+m-1}\cdot R_k^{(\sum^{k-1}_{j=1}n_j)/2+m},
\end{align*}
and

\begin{align*}
 R_{k+1}&\geq  \frac{1}{2}\ld^{m^{*}}\cdot(2R_k)^{m}\cdot(2^{n_k-1}-1)\cdot\bigg[\prod^{k-1}_{j=1}\bigg(\frac{1}{2}\bigg(\frac{2R_k}{R_j}\bigg)^{n_j}-1\bigg)\bigg]\nonumber\\
&\geq\frac{1}{2}\ld^{m^{*}}\cdot(2R_k)^{m}\cdot(2^{n_k-2})\cdot\bigg[\prod^{k-1}_{j=1}\bigg(\frac{1}{4}\bigg(\frac{2R_k}{R_j}\bigg)^{n_j}\bigg)\bigg]\nonumber\\
&=\frac{1}{2} \ld^{m^{*}}\cdot \bigg(\frac{1}{4}\bigg)^{k}\cdot (2R_{k})^{m+\sum^{k-1}_{j=1}n_{j}} \cdot 2^{n_{k}}\cdot\bigg[\prod^{k-1}_{j=1}{R_j}^{-n_j}\bigg]\\
&= \ld^{m^{*}}\cdot 2^{m-1+\sum^{k}_{j=1}(n_{j}-2)} \cdot  R_{k}^{m+\sum^{k-1}_{j=1}n_{j}}\cdot\bigg[\prod^{k-1}_{j=1}{R_j}^{-n_j}\bigg].
\end{align*}
On the other hand, one has
\begin{align*}
 R_{k+1}&=\max_{|z|=2R_k}|f_k(z)|\\
&\leq \max_{|z|=2R_k}|F_0(z)|\cdot\bigg[ \prod^{k}_{j=1}\max_{|z|=2R_k}|F_j(z)|\bigg]\\
&\leq \frac{3}{2}\ld^{m^{*}}\cdot(2R_k)^{m} \cdot \bigg[\prod^{k}_{j=1}\bigg(\frac{1}{2}\bigg(\frac{2R_k}{R_j}\bigg)^{n_j}+1\bigg)\bigg]\\
&\leq \frac{3}{2}\ld^{m^{*}}\cdot(2R_k)^{m} \cdot \bigg[\prod^{k}_{j=1}\bigg(\bigg(\frac{2R_k}{R_j}\bigg)^{n_j}\bigg)\bigg]\\
&\leq  \frac{3}{2}\ld^{m^{*}}\cdot(2R_k)^{m+\sum^k_{j=1}n_j}\cdot\bigg[\prod^{k}_{j=1} R^{-n_j}_j\bigg].
\end{align*}

\end{proof}

\begin{corollary}\label{degreeest-8-4-1}
Suppose the assumption (A*) holds, for the above $\{R_k\}$, one has
\beq\label{inequ-19}
m_{k-1}=m+\sum^{k-1}_{j=1}n_j\leq2\frac{\log R_{k+1}}{\log R_{k}}
\eeq
and
\beq\label{inequ-20}
m_{k}=m+\sum^{k}_{j=1}n_j<\frac{\log R_{k+1}}{\log 2}.
\eeq
\end{corollary}

\begin{proof}
By \eqref{inequ-3}, one has
\beq
R_{k+1}\geq\ld^{m^{*}}\cdot2^{(\sum^{k-1}_{j=1}(2n_j-2))+(n_k-2)+m-1}\cdot R_k^{(\sum^{k-1}_{j=1}n_j)/2+m},
\eeq
so,
\begin{align*}
\log R_{k+1}&\geq \log\ld^{m^{*}}+\log 2^{(\sum^{k-1}_{j=1}(2n_j-2))+(n_k-2)+m-1}+\log R_k^{(\sum^{k-1}_{j=1}n_j)/2+m}\\
&\geq\bigg (\frac{1}{2}\bigg(\sum^{k-1}_{j=1}n_j\bigg)+m\bigg)\cdot\log R_k
\geq\frac{1}{2}\bigg (\bigg(\sum^{k-1}_{j=1}n_j\bigg)+m\bigg)\cdot\log R_k,
\end{align*}
implying that \eqref{inequ-19} holds.

Further, by \eqref{valueest-8} and \eqref{inequ-1}, one has
\begin{align*}
R_{k+1}&\geq\ld^{m^{*}}\cdot 2^{m-1+\sum^{k}_{j=1}(n_{j}-2)} \cdot  R_{k}^{m+\sum^{k-1}_{j=1}n_{j}}\cdot\bigg[\prod^{k-1}_{j=1}{R_j}^{-n_j}\bigg]\\
=&\ld^{m^{*}} \cdot 2^{m+\sum^{k}_{j=1}n_{j}} \cdot\frac{R_k^m}{2\cdot 4^k}\cdot\bigg[\prod^{k-1}_{j=1}\bigg(\frac{R_k}{R_j}\bigg)^{n_j}\bigg]\\
>& 2^{m+\sum^{k}_{j=1}n_{j}}.
\end{align*}
So,  \eqref{inequ-20} holds.
\end{proof}

\begin{lemma}\label{seriescon-8-4-2}
Given any positive real number $\al$. For any $k\geq1$, one has
\begin{align*}
\sum_{q\geq1} \frac{2^q(\log R_{k+q-1})^2}{R^{\alpha}_{k+q-1}}<+\infty.
\end{align*}
Further, the sum of the series tends to zero as $R\to+\infty$.
\end{lemma}

\begin{proof}
Let $a_{q}=\tfrac{2^q(\log R_{k+q-1})^2}{R^{\alpha}_{k+q-1}}$, $q\geq1$. Applying the ratio test,  we show
\begin{align*}
\frac{a_{q+1}}{a_q}=2\cdot\frac{(\log R_{k+q})^2}{(\log R_{k+q-1})^2} \cdot\frac{R^{\al}_{k+q-1}}{R^{\al}_{k+q}}<1.
\end{align*}

By \eqref{valueest-11}, one has
\begin{align*}
&\frac{\log R_{k+q}}{\log R_{k+q-1}}\leq\frac{\log\big(\frac{3}{2}\cdot\ld^{m^{*}}\cdot2^{m+\sum^{k+q-1}_{j=1}n_j}\cdot(R_{k+q-1})^{m+\sum^{k+q-2}_{j=1}n_j}\big)}{\log R_{k+q-1}}\\
\leq&\frac{\log\big(\frac{3}{2}\ld^{m^{*}}\big)}{\log R_{k+q-1}}+\frac{\big(m+\sum^{k+q-1}_{j=1}n_j\big)\cdot\log2}{\log R_{k+q-1}}+\frac{\big(m+\sum^{k+q-2}_{j=1}n_j\big)\cdot\log(R_{k+q-1})}{\log R_{k+q-1}}\\
\leq&2\bigg(m+\sum^{k+q-2}_{j=1}n_j\bigg)+\frac{\big(m+\sum^{k+q-1}_{j=1}n_j\big)\cdot\log2}{\log R_{k+q-1}}. ({\color{red}\star})
\end{align*}

By \eqref{inequ-3}, one has
\begin{align*}
 \frac{R_{k+q-1}}{R_{k+q}}\leq\ld^{-m^{*}}\cdot2^{-[(\sum^{k+q-2}_{j=1}(2n_j-2))+(n_{k+q-1}-2)+m-1]}\cdot R_{k+q-1}^{-(\sum^{k+q-2}_{j=1}n_j)/2-m+1}.
\end{align*}

So,
\begin{align*}
&\frac{a_{q+1}}{a_q}\leq 8\bigg(m+\sum^{k+q-2}_{j=1}n_j\bigg)^2
\cdot R_{k+q-1}^{-\al[(\sum^{k+q-2}_{j=1}n_j)/2+m-1]}\\
&+4\bigg(\frac{\big(m+\sum^{k+q-1}_{j=1}n_j\big)\log2}{\log R_{k+q-1}}\bigg)^2\cdot2^{-\al[(\sum^{k+q-2}_{j=1}(2n_j-2))+(n_{k+q-1}-2)+m-1]}\\ &\times R_{k+q-1}^{-\al[(\sum^{k+q-2}_{j=1}n_j)/2+m-1]}\\
\leq& 8\bigg(m+\sum^{k+q-2}_{j=1}n_j\bigg)^2
\cdot (R_{k+q-1}^{-\al/2})^{(m+\sum^{k+q-2}_{j=1}n_j)}\\
&+16\bigg(\frac{\big(m+\sum^{k+q-1}_{j=1}n_j\big)}{\log R_{k+q-1}}\bigg)^2\cdot2^{-(\al/2)[m+\sum^{k+q-1}_{j=1}n_j]}\cdot (R_{k+q-1}^{-\al/2})^{(m+\sum^{k+q-2}_{j=1}n_j)}\to0,
\end{align*}
since $\lim_{n\to\infty}\tfrac{n^2}{x^n}=0$ for any $x>1$, the ratio test is satisfied and the sum is convergent for sufficiently large $R>1$.

Since $\sum^{\infty}_{n=1}\frac{n^2}{x^n}\leq x^2\sum^{\infty}_{n=1}\frac{n(n+1)}{x^{n+2}}$ for $x>1$, and $(\tfrac{1}{x^n})^{\prime\prime}=\tfrac{n(n+1)}{x^{n+2}}$, one has that $\sum^{\infty}_{n=1}\frac{n^2}{x^n}\leq\tfrac{2x^2}{(x-1)^3}$ for $x>1$.
Hence, the sum tends to zero as $R\to\infty$.
\end{proof}

\subsection{Geometry of Chebyshev polynomial $T_2(z)=2z^2-1$}

In this subsection, the geometric structure of  $T_2(z)=2z^2-1$ is studied. This part follows the main idea of Section 10 in Bishop's work \cite{Bishop2018}. The main idea is the function $F_k$ for $z\in A_k$ can be written in the form of  $C'\cdot T_2(z^{l_{*}})\cdot z^{l_{**}}$, where $C'$ is a constant, $l_{*}$ and $l_{**}$ are two integers, which are dependent on $F_k$ (see \eqref{valueest-3} and \eqref{equ2021-8-26-1}). This is useful in the understanding of the geometric structure of the Fatou and Julia sets.

Denote by $z_2=-1/\sqrt{2}$ the left root of $T_2$, $w_2=0$ the critical point of $T_2$, $\Om_2$ the component of $\{z:\ |T_2(z)|<1\}$ containing $z_2$. Set
\beqq
r_2:=\mbox{dist}(z_2,-1)=1-\frac{1}{\sqrt{2}},\ \widetilde{r}_2:=\mbox{dist}(z_2,w_2)=\frac{1}{\sqrt{2}},
\eeqq
\beqq
D_2:=D(z_2,r_2)=D\bigg(-\frac{1}{\sqrt{2}},1-\frac{1}{\sqrt{2}}\bigg),\ \widetilde{D}_2:=D(z_2,\widetilde{r}_2)=D\bigg(-\frac{1}{\sqrt{2}},\frac{1}{\sqrt{2}}\bigg),
\eeqq
where $D(z_2,r_2)$ is a ball with center $z_2$ and radius $r_2$, $\widetilde{D}_2$ is defined similarly.

\begin{figure}
\begin{center}
\begin{tikzpicture}

\node at (0,0) {$\maltese$};
\node at (0.3,-0.3) {$w_2$};
\node at  (-2.3,0) {$\bullet$};
\node at  (-2.3,-0.3) {$z_2$};

\draw (2, 0) circle (2);
\draw (-2, 0) circle (2);

\draw[dashed] (-3, 0) circle (3);

\draw[dashed] (-2.3, 0) circle (0.9);

\node at  (-3,0) {$D_2$};

\node at  (-6.2,0) {$\widetilde{D}_2$};

\node at  (-0.8,0) {$\Om_2$};

\end{tikzpicture}
\end{center}
\caption{An illustration diagram of $D_2$ and $\widetilde{D}_2$, where the black dot is $z_2$, $\maltese$ represents $w_2$, the real black curve represents
$\{z:\ |T_2(z)|=1\}$, the dashed lines represent $D_2$ and $\widetilde{D}_2$, $\Om_2$ lies in the left lobe of the curve and between $D_2$ and $\widetilde{D}_2$.}\label{fatoucompqua}
\end{figure}

\begin{lemma}\cite[Lemma 9.1]{Bishop2018} For the polynomial $T_2(z)$, one has that
$|T_2|\geq1$ on $\partial\widetilde{D}_2$ and $|T_2|\leq1$ on $\partial D_2$. Thus, $D_2\subset\Om_2\subset\widetilde{D}_2$.
\end{lemma}
Let $\widetilde{m}$ be a positive integer. Set
\beq\label{valueest-3}
H_{\widetilde{m}}(z):=-T_2(\widetilde{r}_2z^{\widetilde{m}}+z_2)=-T_2(\tfrac{1}{\sqrt{2}}z^{\widetilde{m}}-\tfrac{1}{\sqrt{2}})=z^{\widetilde{m}}(2-z^{\widetilde{m}}).
\eeq
The derivative is
\beqq
H^{\prime}_{\widetilde{m}}(z)=\widetilde{m}z^{{\widetilde{m}}-1}(2-z^{\widetilde{m}})+z^{\widetilde{m}}(-{\widetilde{m}}z^{{\widetilde{m}}-1})=2{\widetilde{m}}z^{\widetilde{m}-1}(1-z^{\widetilde{m}}),
\eeqq
this means that all the non-zero critical points are on the unit circle.

\begin{definition}\label{petal731-1}
The complement of the level curve $\ga_{\widetilde{m}}=\{z:\ |H_{\widetilde{m}}(z)|=1\}$ is an open set, denoted by $\Om_{\widetilde{m}}=\mathbb{C}\setminus\ga_{\widetilde{m}}$, with $\widetilde{m}+2$ connected components, a central component containing $0$ is denoted by $\Om^0_{\widetilde{m}}$, an unbounded component containing infinity is denoted by $\Om^{\infty}_{\widetilde{m}}$ , and $\widetilde{m}$ other bounded components are called the petals of $\Om_{\widetilde{m}}$. There exists one and only one critical point on each petal, the union of these $\widetilde{m}$ petals is denoted by $\Om^p_{\widetilde{m}}$.
\end{definition}

\begin{remark}\label{petal-8-6-1}
$H_{\widetilde{m}}$ is an $\widetilde{m}$-to-$1$ branched covering map from $\Om^0_{\widetilde{m}}$ to $\mathbb{D}$ with a single critical point at the origin, and is conformal from the interior of each petal to $\mathbb{D}$.
\end{remark}

\begin{figure}
\begin{center}
\begin{tikzpicture}
\draw (0, 0) circle (2);

\foreach \j in {1,2,3,4,5}
{
 \node at  (72*\j:2.2) {$\j$};
  \draw (72*\j:2.3) circle (0.3);
};

\end{tikzpicture}
\end{center}
\caption{An illustration diagram of the level of the form $\{z:\ |T_2(z^5)|=1\}$}\label{fatoucomplevel5}
\end{figure}

\begin{figure}
\begin{center}
\begin{tikzpicture}
\draw (0, 0) circle (2);

\foreach \j in {1,2,3,...,10}
{
  \node at  (36*\j:2.2) {$\j$};
  \draw (36*\j:2.3) circle (0.3);
};

\end{tikzpicture}
\end{center}
\caption{An illustration diagram of the level of the form $\{z:\ |T_2(z^{10})|=1\}$}\label{fatoucomplevel10}
\end{figure}

\begin{lemma}\cite[Lemma 9.2]{Bishop2018}\label{inequ-17}
 \beqq
\bigg\{z:\ |z|<1-\frac{1}{\widetilde{m}}\bigg\}\subset\Om^0_{\widetilde{m}}\subset\mathbb{D}.
\eeqq
\end{lemma}

\begin{lemma}\cite[Lemma 9.3]{Bishop2018}\label{inequ-18}
\beqq
\bigg\{z:\ |z|>1-\frac{1}{\widetilde{m}}\bigg\}\supset\Om^{\infty}_{\widetilde{m}}\supset\bigg\{z:\ |z|>1+\frac{2}{{\widetilde{m}}}\bigg\},\ \widetilde{m}\geq2.
\eeqq
\end{lemma}

\begin{corollary}\label{inequ-6}
The petal components are contained in the region:
\beq\label{inequ-5}
\Om^{p}_{\widetilde{m}}\subset\bigg\{z:\ |z|\geq1-\frac{1}{\widetilde{m}}\bigg\}\bigcap\bigg\{z:\ |z|\leq1+\frac{2}{{\widetilde{m}}}\bigg\},\ \widetilde{m}\geq2.
\eeq
\end{corollary}

\begin{lemma}
Suppose (A*) holds. For $z\in A_k$, one has
\beq\label{valueest-4}
f(z)=C_k\cdot z^{s_k}\cdot\bigg(H_{n_k}\cdot\bigg(\frac{z}{R_k}\bigg)\bigg)\cdot(1+O(R^{-1}_k)),
\eeq
where
\beq\label{valueest-5}
C_k=\ld^{m^{*}}\cdot (-1)^{k-1}\cdot \bigg(\frac{1}{2}\bigg)^{k}\cdot R_k^{n_k}\cdot\bigg(\prod^{k-1}_{j=1} R^{-n_j}_j\bigg)
\eeq
and
\beq\label{exp8-20-1}
s_k=-n_k+\bigg(m+\sum^{k-1}_{j=1}n_j\bigg)=m_{k-1}-n_k.
\eeq
\end{lemma}

\begin{proof}
Rewrite
\beqq
f(z)=\prod^{\infty}_{k=0}F_k(z)=F_0(z)\cdot \bigg[\prod^{k-1}_{j=1}F_j(z)\bigg]\cdot F_k(z) \cdot\bigg[ \prod^{\infty}_{j=k+1}F_j(z)\bigg]
\eeqq
and
\begin{align*}
 f_k(z)=F_0(z)\cdot\bigg[\prod^{k-1}_{j=1}F_j(z)\bigg]\cdot F_k(z)=z^{-m}\cdot F_0(z)\cdot\bigg[\prod^{k-1}_{j=1}z^{-n_j}F_j(z)\bigg]\cdot z^{m+\sum^{k-1}_{j=1}n_j}\cdot F_k(z).
\end{align*}
So, by \eqref{critical-1}, one has, for $z\in A_k$,
\beqq
z^{-m}\cdot F_0(z)=\ld^{m^{*}}(1+O(R^{-1}_k)).
\eeqq
By \eqref{valueest-2}, one has, for $z\in A_k$,
\begin{align*}
&\prod^{k-1}_{j=1}(z^{-n_j}\cdot F_j(z))
=\prod^{k-1}_{j=1}\bigg( \frac{1}{(-2)}\cdot R^{-n_j}_j\cdot\bigg(1+O\bigg(\bigg(\frac{R_j}{R_k}\bigg)^{n_j}\bigg)\bigg)\bigg)\\
=&\frac{1}{(-2)^{k-1}}\cdot\bigg[ \prod^{k-1}_{j=1} R^{-n_j}_j\bigg]\cdot
\bigg[\prod^{k-1}_{j=1}\bigg(1+O\bigg(\bigg(\frac{R_j}{R_k}\bigg)^{n_j}\bigg)\bigg)\bigg]\\
=&\frac{1}{(-2)^{k-1}}\cdot\bigg[ \prod^{k-1}_{j=1} R^{-n_j}_j\bigg]\cdot (1+2R_k^{-n_0/2}).
\end{align*}
By computation, one has, for $z\in A_k$,
\begin{align}\label{equ2021-8-26-1}
F_k(z)=&\bigg(1-\frac{1}{2}\bigg(\frac{z}{R_k}\bigg)^{n_k}\bigg)
=\bigg(\frac{R_k}{z}\bigg)^{n_k}\cdot\bigg(\frac{z}{R_k}\bigg)^{n_k}\cdot\bigg(1-\frac{1}{2}\bigg(\frac{z}{R_k}\bigg)^{n_k}\bigg)\nonumber\\
=&\bigg(\frac{R_k}{z}\bigg)^{n_k}\cdot\bigg[\bigg(\frac{z}{R_k}\bigg)^{n_k}\cdot\bigg(1-\frac{1}{2}\bigg(\frac{z}{R_k}\bigg)^{n_k}\bigg)\bigg]
=
\bigg(\frac{R_k}{z}\bigg)^{n_k}\cdot\bigg(\frac{1}{2}\bigg)\cdot\bigg(H_{n_k}\bigg(\frac{z}{R_k}\bigg)\bigg),
\end{align}
where $H_{n_k}$ is specified in \eqref{valueest-3}, implying that
\begin{align*}
&z^{m+\sum^{k-1}_{j=1}n_j}\cdot F_k(z)=z^{m+\sum^{k-1}_{j=1}n_j} \cdot\bigg(\frac{R_k}{z}\bigg)^{n_k}\cdot\bigg(\frac{1}{2}\bigg)\cdot\bigg(H_{n_k}\bigg(\frac{z}{R_k}\bigg)\bigg)\\
=&R_k^{n_k}\cdot \bigg(\frac{1}{z}\bigg)^{n_k-(m+\sum^{k-1}_{j=1}n_j)}\cdot\bigg(\frac{1}{2}\bigg)\cdot\bigg(H_{n_k}\bigg(\frac{z}{R_k}\bigg)\bigg)\\
=&R_k^{n_k}\cdot z^{s_k}\cdot\bigg(\frac{1}{2}\bigg)\cdot\bigg(H_{n_k}\bigg(\frac{z}{R_k}\bigg)\bigg).
\end{align*}

These discussions, together with \eqref{valueest-1}, yield that \eqref{valueest-4} holds.
\end{proof}

\begin{remark}\label{geofact-8-6-2}
Suppose $n_k-(m+\sum^{k-1}_{j=1}n_j)>0$. By Definition \ref{petal731-1} and Remark \ref{petal-8-6-1},
the map $\big(\frac{1}{z}\big)^{n_k-(m+\sum^{k-1}_{j=1}n_j)}\cdot H_{n_k}\big(\frac{z}{R_k}\big)$ in \eqref{valueest-4}
 is an $(2n_k-(m+\sum^{k-1}_{j=1}n_j))$-to-$1$ branched covering map from $R_k\cdot\Om^0_{n_k}$ to $\mathbb{D}$ with a single critical point at the origin, and is conformal from the interior of each petal to $\mathbb{D}$.  If $n_k-(m+\sum^{k-1}_{j=1}n_j)<0$, then similar conclusions hold for smaller subsets of $R_k\cdot\Om^0_{n_k}$ and $R_k\cdot\Om^p_{n_k}$. This geometric fact will be repeated used in the following discussions of the structure of the Julia and Fatou sets.
\end{remark}

\begin{lemma}\label{valueest-6}
For $|z|=\tau R_k$, where $\tau>1$ is a positive constant, one has
\begin{itemize}
\item[(i)]
 \beq\label{valueest-7}
\bigg|z^{s_k}\cdot\bigg(H_{n_k}\bigg(\frac{z}{R_k}\bigg)\bigg)\bigg|
\leq4\cdot R_k^{s_k}\cdot\tau^{m+\sum^{k}_{j=1}n_j};
\eeq
\item[(ii)] further, if $\tau^{n_k}\geq4$, then
\beq\label{valueest-10}
\bigg|z^{s_k}\cdot\bigg(H_{n_k}\bigg(\frac{z}{R_k}\bigg)\bigg)\bigg|
\geq \bigg(\frac{1}{2}\bigg)\cdot R_k^{s_k}\cdot\tau^{m+\sum^{k}_{j=1}n_j}.
\eeq
\end{itemize}
\end{lemma}

\begin{remark}
In Case (ii), if $\tau^{n_k}<4$, then it is possible that $\big|\big(\frac{1}{z}\big)^{n_k-(m+\sum^{k-1}_{j=1}n_j)}\cdot\big(H_{n_k}\big(\frac{z}{R_k}\big)\big)\big|=0$, since we might meet the zeros of $H_{n_k}$.
\end{remark}

\begin{proof}
{\bf Case (i)} Consider the situation $\tau>1$. By direct computation, one has
 \begin{align*}
 &\bigg|z^{s_k}\cdot\bigg(H_{n_k}\bigg(\frac{z}{R_k}\bigg)\bigg)\bigg|\\
\leq&({\tau R_k})^{s_k}\cdot\bigg[\bigg(\frac{\tau R_k}{R_k}\bigg)^{n_k}\cdot\bigg(2+\bigg(\frac{\tau R_k}{R_k}\bigg)^{n_k}\bigg)\bigg]\\
=&({\tau R_k})^{s_k}\cdot[\tau^{n_k}\cdot(2+\tau^{n_k})]\\
\leq&({\tau R_k})^{s_k}\cdot[(\tau^{n_k}+1)^2]\\
\leq&({\tau R_k})^{s_k}\cdot[(2\tau^{n_k})^2]
=({\tau R_k})^{s_k}\cdot 4\cdot\tau^{2n_k}\\
=&4\cdot R_k^{s_k}\cdot\tau^{m+\sum^{k}_{j=1}n_j}.
\end{align*}

{\bf Case (ii)} For $\tau>1$ and $\tau^{n_k}\geq4$, one has
\begin{align*}
 &\bigg|z^{s_{k}}\cdot\bigg(H_{n_k}\bigg(\frac{z}{R_k}\bigg)\bigg)\bigg|\\
\geq&({\tau R_k})^{s_k}\cdot\bigg[\bigg(\frac{\tau R_k}{R_k}\bigg)^{n_k}\cdot\bigg(\bigg(\frac{\tau R_k}{R_k}\bigg)^{n_k}-2\bigg)\bigg]\\
=&({\tau R_k})^{s_k}\cdot[\tau^{n_k}\cdot(\tau ^{n_k}-2)]\\
\geq&({\tau R_k})^{s_k}\cdot\bigg[\tau^{n_k}\cdot\bigg(\frac{1}{2}\tau^{n_k}\bigg)\bigg]\\
=&\bigg(\frac{1}{2}\bigg)\cdot R_k^{s_k}\cdot\tau^{m+\sum^{k}_{j=1}n_j}.
\end{align*}
\end{proof}

\begin{lemma}\label{valueest-6-8-26-2}
For $|z|=\tau R_k$, where $\tau<1$ is a positive constant, one has
 \beq\label{valueest-15}
\bigg|z^{s_k}\cdot\bigg(H_{n_k}\bigg(\frac{z}{R_k}\bigg)\bigg)\bigg|
\leq3\cdot R_k^{s_k}\cdot\tau^{m+\sum^{k-1}_{j=1}n_j}
\eeq
and
\beq\label{valueest-13}
\bigg|z^{s_{k}}\cdot\bigg(H_{n_k}\bigg(\frac{z}{R_k}\bigg)\bigg)\bigg|\geq
R_k^{s_k}\cdot\tau^{m+\sum^{k-1}_{j=1}n_j}.
\eeq
\end{lemma}

\begin{proof}
Direct calculation tells us that
 \begin{align*}
 &\bigg|z^{s_k}\cdot\bigg(H_{n_k}\bigg(\frac{z}{R_k}\bigg)\bigg)\bigg|\\
\leq&({\tau R_k})^{s_k}\cdot\bigg[\bigg(\frac{\tau R_k}{R_k}\bigg)^{n_k}\cdot\bigg(2+\bigg(\frac{\tau R_k}{R_k}\bigg)^{n_k}\bigg)\bigg]\\
\leq&({\tau R_k})^{s_k}\cdot[3\tau^{n_k}]\\
=&3\cdot R_k^{s_k}\cdot\tau^{m+\sum^{k-1}_{j=1}n_j},
\end{align*}
and
\begin{align*}
 &\bigg|z^{s_{k}}\cdot\bigg(H_{n_k}\bigg(\frac{z}{R_k}\bigg)\bigg)\bigg|\\
\geq&({\tau R_k})^{s_k}\cdot\bigg[\bigg(\frac{\tau R_k}{R_k}\bigg)^{n_k}\cdot\bigg(2-\bigg(\frac{\tau R_k}{R_k}\bigg)^{n_k}\bigg)\bigg]\\
\geq&({\tau R_k})^{s_k}\cdot\tau^{n_k}\\
=&R_k^{s_k}\cdot\tau^{m+\sum^{k-1}_{j=1}n_j}.
\end{align*}
\end{proof}

\begin{lemma}\label{est-1}
Suppose (A*) holds.
For $\tfrac{3}{2}R_k\leq|z|\leq 4R_k$, the function $f(z)$ can be written as
\beq\label{inequ-21}
f(z)= C^*_k\cdot z^{(m+\sum^{k}_{j=1}n_j)}\cdot\bigg(1+O\bigg(\bigg(\frac{2}{3}\bigg)^{n_k}\bigg)\bigg)\cdot(1+O(R^{-1}_k)),
\eeq
where
\beqq\label{inequ-23}
C^*_k=\ld^{m^{*}}\cdot \bigg(-\frac{1}{2}\bigg)^{k} \cdot\bigg[\prod^{k}_{j=1} R^{-n_j}_j\bigg]
\eeqq
is specified in \eqref{valueest-5}.
\end{lemma}

\begin{proof}
By direct computation,  for $\tfrac{3}{2}R_k\leq|z|\leq 4R_k$, one has
\begin{align*}
&\bigg[\bigg(\frac{z}{R_k}\bigg)^{n_k}\cdot\bigg(2-\bigg(\frac{z}{R_k}\bigg)^{n_k}\bigg)\bigg]=
z^{2n_k}\cdot R^{-2n_k}_k\cdot \bigg(2\bigg(\frac{R_k}{z}\bigg)^{n_k}-1\bigg)\\
=&(-1)\cdot z^{2n_k}\cdot R^{-2n_k}_k\cdot \bigg(1-2\bigg(\frac{R_k}{z}\bigg)^{n_k}\bigg)
=(-1)\cdot z^{2n_k}\cdot R^{-2n_k}_k\cdot \bigg(1+O\bigg(\frac{R_k}{\frac{3}{2}R_k}\bigg)^{n_k}\bigg)\\
=&(-1)\cdot z^{2n_k}\cdot R^{-2n_k}_k\cdot\bigg(1+O\bigg(\bigg(\frac{2}{3}\bigg)^{n_k}\bigg)\bigg).
\end{align*}
This, together with \eqref{valueest-4} and \eqref{valueest-5}, implies \eqref{inequ-21}.
\end{proof}

In \eqref{inequ-21}, one has
\begin{align*}
&\bigg(1+O\bigg(\bigg(\frac{2}{3}\bigg)^{n_k}\bigg)\bigg)\cdot(1+O(R^{-1}_k))\\
=&
\bigg(1+O\bigg(\bigg(\frac{2}{3}\bigg)^{n_k}\bigg)\bigg)\cdot(1+O(R^{-1}_k))
=1+O\bigg(\bigg(\frac{2}{3}\bigg)^{n_k}\bigg)+O(R^{-1}_k).
\end{align*}
Set
\beq\label{appep-8-17-2}
\ep_k:=C\cdot\bigg(\bigg(\frac{2}{3}\bigg)^{n_k}+R^{-1}_k\bigg)
=C\cdot\bigg(\frac{1}{\big(\frac{3}{2}\big)^{n_k}}+R^{-1}_k\bigg),
\eeq
where $C$ is a positive constant.
In Lemma \ref{convergrad}, if $L_0=\tfrac{3}{2}$, then $\ep_k$ can be taken as small as we want if $R$ is sufficiently large, and $\sum_{k\geq1}\ep_k$ is convergent.

\begin{lemma}\label{criticalvk-8-6-3}
Suppose (A*) holds. Then, $f^{\prime}$ is non-zero on $V_k$ for $k\geq1$.
\end{lemma}

\begin{proof}
It follows from \eqref{inequ-21} that
\beqq
f(z)=C^*_k\cdot z^{m_k}\cdot(1+h_k(z)).
\eeqq
So,
\begin{align*}
f^{\prime}(z)&=(C^*_k\cdot z^{m_k}\cdot (1+h_k(z)))^{\prime}=C^*_k\cdot m_k\cdot z^{m_k-1}\cdot(1+h_k(z))+C^*_k\cdot z^{m_k}\cdot h^{\prime}_k(z)\\
&=C^*_k\cdot z^{m_k-1}\cdot [m_k\cdot(1+h_k(z))+h^{\prime}_k(z)].
\end{align*}
For $z\in V_k$, the above computation gives
\begin{align*}
 |f^{\prime}(z)|\geq& |C^*_k|\cdot |z|^{m_k-1}\cdot [m_k+O(m_k\ep_k)+O(h^{\prime}_k(z))]\\
\geq&
\ld^{m^{*}}\cdot \bigg(\frac{1}{2}\bigg)^{k}\cdot \bigg[\prod^{k}_{j=1} R^{-n_j}_j\bigg]\cdot
 R_k^{m_k-1}\cdot [m_k+O(m_k\ep_k)+O(h^{\prime}_k(z))]\\
\geq&\ld^{m^{*}}\cdot \bigg(\frac{1}{2}\bigg)^{k}\cdot R_k^{m-1}\cdot\bigg[\prod^{k-1}_{j=1}\frac{R^{n_j}_k}{R^{n_j}_j}\bigg]
\cdot [m_k+O(m_k\ep_k)+O(h^{\prime}_k(z))]\\
\geq&\ld^{m^{*}}\cdot R_k^{m-1}\cdot\bigg[\prod^{k-1}_{j=1}\bigg(\frac{R_k}{2R_j}\bigg)^{n_j}\bigg]
\cdot [m_k+O(m_k\ep_k)+O(h^{\prime}_k(z))]\\
\geq&\ld^{m^{*}}\cdot R_k^{m-1}\cdot\bigg[\prod^{k-1}_{j=1}\bigg(\frac{R_k}{2R_j}\bigg)^{n_j}\bigg]>0,
\end{align*}
where $O(h^{\prime}_k(z))$ is estimated by the classical Cauchy formula in Lemma \ref{cauchyformula8-23-1}.
\end{proof}

\subsection{The order of growth}\label{ordergrowth1}

In this subsection, based on the above construction, we show that, for any $s\in(0,+\infty]$, we can pick up a transcendental function from the family constructed above such that the order of growth is $s$.

\begin{theorem}\label{est-2}
There is a transcendental function defined above with infinite order of growth. In particular,
if $n_k=(\lfloor R_k\rfloor)^k$, then the order of growth is $+\infty$, where $\lfloor x\rfloor$ is a function, giving the largest integer less than or equal to $x$ for real $x$.
\end{theorem}

\begin{proof}
By \eqref{valueest-4}, \eqref{valueest-5}, \eqref{valueest-10}, and \eqref{inequ-1}, one has
\begin{align*}
&\min\{|f(z)|:\ |z|=2R_k\}\\
\geq&
 \ld^{m^{*}}\cdot \bigg(\frac{1}{2}\bigg)^{k}\cdot R_k^{n_k}\cdot\bigg[\prod^{k-1}_{j=1} R^{-n_j}_j\bigg]\cdot\bigg[ \bigg(\frac{1}{2}\bigg)\cdot R_k^{s_{k}}\cdot 2^{m+\sum^{k}_{j=1}n_j}\bigg]\cdot\frac{1}{2}\\
=& \ld^{m^{*}}\cdot \bigg(\frac{1}{2}\bigg)^{k+2}\cdot R_k^{m}\cdot\bigg[\prod^{k-1}_{j=1} \bigg(\frac{R_k}{R_j}\bigg)^{n_j}\bigg]\cdot 2^{m+\sum^{k}_{j=1}n_j}\\
\geq&\ld^{m^{*}}\cdot R_k^{m}\cdot\bigg[\prod^{k-1}_{j=1} \bigg(\frac{R_k}{4R_j}\bigg)^{n_j}\bigg]\cdot 2^{m+\sum^{k}_{j=1}n_j}\\
\geq&2^{n_k}.
\end{align*}
If we take $n_k=\lfloor R_k^s\rfloor$,  then the order of growth satisfies
\begin{align}\label{inequ-10}
\rho(f)&=\limsup_{z\to\infty}\frac{\log\log|f(z)|}{\log|z|}\geq \limsup_{k\to\infty}\frac{\log\log2^{\lfloor R_k^s\rfloor}}{\log(2R_k)}=
\limsup_{k\to\infty}\frac{\log{\lfloor R_k^s\rfloor}+\log\log2}{\log2+\log R_k}\nonumber\\
&\geq\limsup_{k\to\infty}\frac{\log\big({\tfrac{R_k^s}{2}}\big)+\log\log2}{\log2+\log R_k}\geq s.
\end{align}
If we take $n_k=(\lfloor R_k\rfloor)^k$, then the order of growth satisfies
\begin{align*}
\rho(f)=&\limsup_{z\to\infty}\frac{\log\log|f(z)|}{\log|z|}\geq \limsup_{k\to\infty}\frac{\log\log2^{(\lfloor R_k\rfloor)^k}}{\log(2R_k)}\\
=&
\limsup_{k\to\infty}\frac{\log{(\lfloor R_k\rfloor)^k}+\log\log2}{\log2+\log R_k}
\geq\limsup_{k\to\infty}\frac{k\log{\lfloor R_k\rfloor}+\log\log2}{\log2+\log R_k}
\geq\infty.
\end{align*}

\end{proof}

\begin{theorem}\label{est-5}
For any $s\in(0,+\infty)$, there is a transcendental function defined above with the order of growth $s$. In particular,  if $n_k=\lfloor R_k^s\rfloor$, then the order of growth is $s$.
\end{theorem}

\begin{proof}
By \eqref{inequ-10}, it is sufficient to give an upper bound for the order of growth.

By \eqref{valueest-4}, \eqref{valueest-5}, and \eqref{valueest-7}, one has
\begin{align*}
&\max\{|f(z)|:\ |z|=2R_k\}\\
\leq&
 \ld^{m^{*}}\cdot \bigg(\frac{1}{2}\bigg)^{k}\cdot R_k^{n_k}\cdot\bigg[\prod^{k-1}_{j=1} R^{-n_j}_j\bigg]\cdot\bigg[ 4\cdot R_k^{s_k}\cdot2^{m+\sum^{k}_{j=1}n_j}\bigg]\cdot2\\
=& \ld^{m^{*}}\cdot \bigg(\frac{1}{2}\bigg)^{k-3}\cdot R_k^{m}\cdot\bigg[\prod^{k-1}_{j=1} \bigg(\frac{R_k}{R_j}\bigg)^{n_j}\bigg]\cdot 2^{m+\sum^{k}_{j=1}n_j}\\
\leq& \ld^{m^{*}}\cdot R_k^{m}\cdot\bigg[\prod^{k-1}_{j=1} \bigg(\frac{R_k}{R_j}\bigg)^{n_j}\bigg]\cdot 2^{m+\sum^{k}_{j=1}n_j}.
\end{align*}
So,
\begin{align*}
&\frac{\log\log|f(z)|}{\log|z|}\leq\frac{\log\log\bigg( \ld^{m^{*}}\cdot R_k^{m}\cdot\bigg[\prod^{k-1}_{j=1} \big(\frac{R_k}{R_j}\big)^{n_j}\bigg]\cdot 2^{m+\sum^{k}_{j=1}n_j}\bigg)}{\log2R_k}\\
=&\frac{\log\bigg[\log( \ld^{m^{*}})+\log( R_k^{m})+\log\big(\prod^{k-1}_{j=1} \big(\frac{R_k}{R_j}\big)^{n_j}\big)+\log(2^{m+\sum^{k}_{j=1}n_j})\bigg]}{\log2R_k}.
\end{align*}
Using the inequality $\log(x+y)\leq\log x+\log y$ for $x,y\geq2$, this is bounded by
\begin{align*}
\frac{\log\log( \ld^{m^{*}})}{\log2R_k}+
\frac{\log\log( R_k^{m})}{\log2R_k}+
\frac{\log\log\big(\prod^{k-1}_{j=1} \big(\frac{R_k}{R_j}\big)^{n_j}\big)}{\log2R_k}+
\frac{\log\log(2^{m+\sum^{k}_{j=1}n_j})}{\log2R_k}.
\end{align*}
So, simple calculation gives us
\beqq
\lim_{k\to\infty}\frac{\log\log( \ld^{m^{*}})}{\log2R_k}=
\lim_{k\to\infty}\frac{\log m^{*}+\log\log\ld}{\log2+\log R_k}=0,
\eeqq
\beqq
\lim_{k\to\infty}\frac{\log\log( R_k^{m})}{\log2R_k}=\lim_{k\to\infty}\frac{\log m+\log\log R_k}{\log2+\log R_k}=0.
\eeqq
By \eqref{inequ-20}, one has
\begin{align*}
&\lim_{k\to\infty}\frac{\log\log\big(\prod^{k-1}_{j=1} \big(\frac{R_k}{R_j}\big)^{n_j}\big)}{\log2R_k}
=\lim_{k\to\infty}\frac{\log\big(\sum^{k-1}_{j=1} n_j(\log\big(\frac{R_k}{R_j}\big)\big)\big)}{\log2R_k}\\
\leq&
\lim_{k\to\infty}\frac{\log\big(\big(\sum^{k-1}_{j=1} n_j\big)\log\big(\frac{R_k}{R_1}\big)\big)}{\log2R_k}\leq
\lim_{k\to\infty}\frac{\log\big(2\frac{\log R_k }{\log2}\log\big(\frac{R_k}{R_1}\big)\big)}{\log2R_k}\\
\leq&
\lim_{k\to\infty}\frac{\log\big(\frac{2 }{\log2}\big)}{\log2R_k}
+
\lim_{k\to\infty}\frac{\log\log R_k}{\log2R_k}
+
\lim_{k\to\infty}\frac{\log\big(\log\big(\frac{R_k}{R_1}\big)\big)}{\log2R_k}=0,
\end{align*}
and
\begin{align*}
&\lim_{k\to\infty}\frac{\log\log(2^{m+\sum^{k}_{j=1}n_j})}{\log2R_k}=
\lim_{k\to\infty}\frac{\log\big((m+\sum^{k}_{j=1}n_j)\log2\big)}{\log2R_k}\\
=&
\lim_{k\to\infty}\frac{\log\big((m+\sum^{k-1}_{j=1}n_j+n_k)\log2\big)}{\log2R_k}
\leq\lim_{k\to\infty}\frac{\log\big((2\frac{\log R_k }{\log2}+n_k)\log2\big)}{\log2R_k}\\
\leq&\lim_{k\to\infty}\frac{\log\big((2\log R_k \big)}{\log2R_k}+
\lim_{k\to\infty}\frac{\log\big(n_k\log2\big)}{\log2R_k}=\lim_{k\to\infty}\frac{\log\big(n_k\log2\big)}{\log2R_k}.
\end{align*}
Take $n_k=\lfloor R_k^s\rfloor$, one has
\begin{align*}
&\lim_{k\to\infty}\frac{\log\big(n_k\log2\big)}{\log2R_k}=\lim_{k\to\infty}\frac{\log\big(\lfloor R_k^s\rfloor\log2\big)}{\log2R_k}\leq\lim_{k\to\infty}\frac{\log\big((R_k^s+2)\log2\big)}{\log2R_k}\\
\leq&\lim_{k\to\infty}\frac{\log\big(R_k^s\log2\big)}{\log2R_k}+\lim_{k\to\infty}\frac{\log\big(2\log2\big)}{\log2R_k}
\leq s.
\end{align*}
Therefore, the order of growth is $s$.

\end{proof}

\begin{remark}\label{criteria-Ber}
In the following discussions, we will show that the dimension is $1$.

In the proof of Theorem \ref{est-2}, if $n_k=\lfloor R_k\rfloor$, then
\eqref{est-3} holds.
This gives a first example satisfying \eqref{est-3}, such that the Fatou set has a multiply connected component and the packing dimension of the Julia set in the escaping set is $1$.

Further, using similar discussions in the proof of Theorems \ref{est-2} and \ref{est-5}, one has that if $n_k=\lfloor (\log (R_k))^s\rfloor$, then
\beqq
\liminf_{r\to\infty}\frac{\log\log(\max_{|z|=r}|f(z)|)}{\log\log r}=s,
\eeqq
the Fatou sets of these functions also have multiply  connected components, and the packing dimension of the Julia set in the escaping set is $1$.
\end{remark}

\subsection{The inclusion relationship}\label{inclusion8-26-1}

In this subsection, we show \eqref{equ2021-1-26-2}, that is, $A_{k+1}\subset f(A_k)$ and $f(B_k)\subset B_{k+1}$, where $A_k$ and $B_k$ are specified in \eqref{equ2021-1-26-1}.

Let $A=\{z:\ a\leq|z|\leq b\}$, the inner boundary and outer boundary of $A$ are denoted by $\partial_iA=\{z:\ |z|=a\}$ and $\partial_oA=\{z:\ |z|=b\}$, respectively. The boundary of $A$ is $\partial A=\partial_iA\cup \partial_oA$.

\begin{lemma}\cite[Lemma 11.1]{Bishop2018}\label{valueest-12}
 Suppose $g$ is holomorphic on an annulus $W=\{a<|z|<b\}$ and continuous up to the boundary. Let $U=\{c<|z|<d\}$.
\begin{itemize}
 \item[(1)] Assume $|g(z)|\leq c$ on $\partial_iW$ and $|g(z)|\geq d$ on $\partial_oW$. Then $U\subset g(W)$.
\item[(2)] Suppose that $g$ has no zeros in $W$ and $g(\partial W)\subset \overline{U}$. Then $g(W)\subset\overline{U}$.
\end{itemize}
\end{lemma}

\subsubsection{The estimate $A_{k+1}\subset f(A_k)$}
\par
\begin{definition}
Set
\beq\label{vkregion8-4-5}
V_k:=\bigg\{z:\ \frac{3}{2}R_k\leq|z|\leq\frac{5}{2}R_k\bigg\}\ \text{and}\ U_k:=\bigg\{z:\ \frac{5}{4}R_k\leq|z|\leq3R_{k}\bigg\}.
\eeq
\end{definition}

\begin{lemma}\label{valueest-17}
Suppose (A*) and (A**) hold, one has
$A_{k+1}\subset f(V_k)\subset f(A_k)$. The inner boundary of $V_k$ is mapped into $B_k$, and the outer boundary of $V_k$ is mapped into $B_{k+1}$.
\end{lemma}

\begin{proof}
The inner boundary of $V_k$ is $\partial_iV_k=\{z:\ |z|=\tfrac{3}{2}R_k\}$ and  the outer boundary of $V_k$ is $\partial_oV_k=\{z:\ |z|=\tfrac{5}{2}R_k\}$.

First, we show that the inner boundary of $V_k$ is mapped into $B_k$, that is, $4R_{k}\leq\min_{z\in\partial_iV_k}|f(z)|\leq\max_{z\in\partial_iV_k}|f(z)|\leq\tfrac{1}{4}R_{k+1}$.

We show $\max_{z\in\partial_iV_k}|f(z)|\leq\tfrac{1}{4}R_{k+1}$.

It follows from $m\geq16>2\cdot\tfrac{\log4}{\log4-\log3}\approx9.63768$ and $n_j\geq8>\tfrac{3}{2}\cdot\tfrac{\log4}{\log4-\log3}\approx7.22826$ that
\beqq
\bigg(m+\sum^k_{j=1}n_j\bigg)\cdot\log3\leq(m-2)\cdot\log4+\bigg(\sum^k_{j=1}(n_j-\tfrac{3}{2})\bigg)\cdot\log4,
\eeqq
implying that
\beq\label{valueest-9}
4 \cdot\bigg(\frac{1}{2}\bigg)^{k-1}\cdot\bigg(\frac{3}{2}\bigg)^{m+\sum^{k}_{j=1}n_j}
\leq\frac{1}{4}\cdot 2^{m-1+\sum^{k}_{j=1}(n_j-2)}.
\eeq
By \eqref{valueest-4}, \eqref{valueest-5}, and \eqref{valueest-7} ($\tau=\tfrac{3}{2}$ in Lemma \ref{valueest-6}),  one has,
\begin{align*}
&|f(z)|\\
\leq&\bigg[\ld^{m^{*}}\cdot \bigg(\frac{1}{2}\bigg)^{k}\cdot R_k^{n_k}\cdot\bigg(\prod^{k-1}_{j=1} R^{-n_j}_j\bigg)\bigg]\cdot \bigg[4\cdot R_k^{s_k}\cdot\bigg(\frac{3}{2}\bigg)^{m+\sum^{k}_{j=1}n_j}\bigg]\cdot 2\\
=&\ld^{m^{*}}\cdot4\cdot \bigg(\frac{1}{2}\bigg)^{k-1}\cdot R_k^{m+\sum^{k-1}_{j=1}n_j}\cdot\bigg[\prod^{k-1}_{j=1} R^{-n_j}_j\bigg]\cdot \bigg(\frac{3}{2}\bigg)^{m+\sum^{k}_{j=1}n_j},\\
\leq&\frac{1}{4}R_{k+1},
\end{align*}
where \eqref{valueest-8} and \eqref{valueest-9} are used in the last inequality.

Now, we prove $\min_{z\in\partial_iV_k}|f(z)|\geq4R_{k}$.

It follows from $n_k>4>\tfrac{\log4}{\log(3/2)}\approx3.41982$ and $m>3>\tfrac{\log8}{\log3}+1\approx2.89279$ that $(\tfrac{3}{2})^{n_k}(\tfrac{1}{2})^2\geq1$ and $3^{m-1}\geq8$.
By \eqref{valueest-4}, \eqref{valueest-5}, and \eqref{valueest-10} ($\tau=\tfrac{3}{2}$ in Lemma \ref{valueest-6}), one has

\begin{align*}
&|f(z)|\\
\geq&\bigg[  \ld^{m^{*}}\cdot \bigg(\frac{1}{2}\bigg)^{k}\cdot R_k^{n_k}\cdot\bigg(\prod^{k-1}_{j=1} R^{-n_j}_j\bigg)\bigg]\cdot\bigg[ \bigg(\frac{1}{2}\bigg)\cdot R_k^{s_k}\cdot\bigg(\frac{3}{2}\bigg)^{m+\sum^{k}_{j=1}n_j}\bigg]\cdot\frac{1}{2}\\
\geq&  \ld^{m^{*}}\cdot \bigg(\frac{1}{2}\bigg)^{k+2}\cdot R_k^{m+\sum^{k-1}_{j=1}n_j}\cdot\bigg[\prod^{k-1}_{j=1} R^{-n_j}_j\bigg] \cdot \bigg(\frac{3}{2}\bigg)^{m+\sum^{k}_{j=1}n_j}\\
\geq&4\cdot\frac{3}{2}\ld^{m^{*}}\cdot(2R_{k-1})^{m+\sum^{k-1}_{j=1}n_j}\cdot
\bigg[\prod^{k-1}_{j=1} R^{-n_j}_j\bigg]\geq 4R_k,
\end{align*}
where \eqref{inequ-1} is used in the last but one inequality, and \eqref{valueest-11} is used in the last inequality.

Second, we prove that the outer boundary of $V_k$ is mapped into $B_{k+1}$, that is, $4R_{k+1}\leq\min_{z\in\partial_oV_k}|f(z)|\leq\max_{z\in\partial_oV_k}|f(z)|\leq\tfrac{1}{4}R_{k+2}$.

Now, we verify $\max_{z\in\partial_oV_k}|f(z)|\leq\tfrac{1}{4}R_{k+2}$.

By \eqref{valueest-4}, \eqref{valueest-5}, and \eqref{valueest-7} ($\tau=\tfrac{5}{2}$ in Lemma \ref{valueest-6}), one has
\begin{align*}
&|f(z)|\\
\leq&\bigg[\ld^{m^{*}}\cdot \bigg(\frac{1}{2}\bigg)^{k}\cdot R_k^{n_k}\cdot\bigg(\prod^{k-1}_{j=1} R^{-n_j}_j\bigg)\bigg]\cdot \bigg[4\cdot R_k^{s_k}\cdot\bigg(\frac{5}{2}\bigg)^{m+\sum^{k}_{j=1}n_j}\bigg]\cdot2\\
=&\ld^{m^{*}}\cdot4\cdot \bigg(\frac{1}{2}\bigg)^{k-1}\cdot R_k^{m+\sum^{k-1}_{j=1}n_j}\cdot\bigg[\prod^{k-1}_{j=1} R^{-n_j}_j\bigg]\cdot \bigg(\frac{5}{2}\bigg)^{m+\sum^{k}_{j=1}n_j},\\
\leq&\frac{1}{4}\cdot\bigg[ \ld^{m^{*}}\cdot 2^{m-1+\sum^{k+1}_{j=1}(n_{j}-2)} \cdot  R_{k+1}^{m+\sum^{k}_{j=1}n_{j}}\cdot\bigg(\prod^{k}_{j=1}{R_j}^{-n_j}\bigg)\bigg]\\
\leq&\frac{1}{4}R_{k+2},
\end{align*}
where \eqref{inequ-1} and \eqref{valueest-8} are used.

Finally, we prove that $4R_{k+1}\leq\min_{z\in\partial_oV_k}|f(z)|$.

It follows from $m>13>\tfrac{2\log4}{\log5-\log4}\approx12.4251$ and $n_j>7>\tfrac{\log4}{\log5-\log4}\approx6.21257$ that $5^m\geq2^{2m+4}$ and $5^{n_j}\geq2^{2n_j+2}$.

By \eqref{valueest-4}, \eqref{valueest-5}, and \eqref{valueest-10} ($\tau=\tfrac{5}{2}$ in Lemma \ref{valueest-6}), one has
\begin{align*}
&|f(z)|\\
\geq&\bigg[  \ld^{m^{*}}\cdot \bigg(\frac{1}{2}\bigg)^{\sum^{k}_{j=1}l_{j}}\cdot R_k^{n_kl_{k}}\cdot\bigg(\prod^{k-1}_{j=1} R^{-n_jl_{j}}_j\bigg)\bigg]\cdot\bigg[ \bigg(\frac{1}{2}\bigg)^{l_{k}}\cdot R_k^{s_k}\cdot\bigg(\frac{5}{2}\bigg)^{m+\sum^{k}_{j=1}n_jl_{j}}\bigg]\cdot\frac{1}{2}\\
\geq&  \ld^{m^{*}}\cdot \bigg(\frac{1}{2}\bigg)^{l_k+\sum^{k}_{j=1}l_{j}+1}\cdot R_k^{m+\sum^{k-1}_{j=1}n_jl_{j}}\cdot\bigg[\prod^{k-1}_{j=1} R^{-n_jl_{j}}_j\bigg] \cdot \bigg(\frac{5}{2}\bigg)^{m+\sum^{k}_{j=1}n_jl_{j}}\\
\geq&4\cdot\bigg[\frac{3}{2}\ld^{m^{*}}\bigg]\cdot(2R_{k})^{m+\sum^{k}_{j=1}n_jl_{j}}\cdot\bigg[\prod^{k}_{j=1} R^{-n_jl_{j}}_j\bigg]\geq 4R_{k+1},
\end{align*}
where \eqref{inequ-1} is used in the last inequality.

Hence, $A_{k+1}\subset f(V_k)$ by the first part of Lemma \ref{valueest-12}.

\end{proof}

\subsubsection{$f(B_k)\subset B_{k+1}$}

\begin{lemma}\label{boundaryfatou}
Suppose (A*) and (A**) hold, one has
\begin{itemize}
\item $f(\partial_oA_k)\subset B_{k+1}$, where $\partial_oA_k=\{z:\ |z|=4R_k\}$ is the  outer boundary of $A_k$;
\item $f(\partial_iA_k)\subset B_k$, where $\partial_iA_k=\{z:\ |z|=R_k/4\}$ is the inner boundary of $A_k$.
\end{itemize}
\end{lemma}

\begin{proof}
We will show $f(\partial_oA_k)\subset B_{k+1}$, that is, $$4R_{k+1}\leq\min_{z\in\partial_oA_k}|f(z)|\leq\max_{z\in\partial_oA_k}|f(z)|\leq\tfrac{1}{4}R_{k+2}.$$

First, we prove $\min_{z\in\partial_oA_k}|f(z)|\geq4R_{k+1}$.

By \eqref{valueest-4}, \eqref{valueest-5}, and \eqref{valueest-10} ($\tau=4$ in Lemma \ref{valueest-6}), one has, for $m\geq16$ and $n_j\geq8$,
\begin{align*}
&|f(z)|\\
=&\bigg|C_k \cdot(4R_k)^{s_k}\cdot\bigg(H_{n_k}\bigg(\frac{z}{R_k}\bigg)\bigg)\cdot(1+O(R^{-1}_k))\bigg|\\
\geq&\bigg[\ld^{m^{*}}\cdot \bigg(\frac{1}{2}\bigg)^{k}\cdot R_k^{n_k}\cdot\bigg(\prod^{k-1}_{j=1} R^{-n_j}_j\bigg)\bigg]\cdot\bigg[ \bigg(\frac{1}{2}\bigg)\cdot R_k^{s_k}\cdot4^{m+\sum^{k}_{j=1}n_j}\bigg]\cdot\frac{1}{2}\\
=&  \ld^{m^{*}}\cdot \bigg(\frac{1}{2}\bigg)^{k+2}\cdot R_k^{m+\sum^{k-1}_{j=1}n_j}\cdot\bigg[\prod^{k-1}_{j=1} R^{-n_j}_j\bigg] \cdot 4^{m+\sum^{k}_{j=1}n_j}\\
\geq&4\cdot\frac{3}{2}\cdot\ld^{m^{*}}\cdot(2R_{k})^{m+\sum^{k}_{j=1}n_j}\cdot\bigg[\prod^{k}_{j=1} R^{-n_j}_j\bigg]\geq 4R_{k+1},
\end{align*}
where \eqref{valueest-11} is used in the last but one inequality.

Second, we show $\max_{z\in\partial_oA_k}|f(z)|\leq\tfrac{1}{4}R_{k+2}$.

By \eqref{valueest-4}, \eqref{valueest-5}, and \eqref{valueest-7} ($\tau=4$ in Lemma \ref{valueest-6}), for $m\geq16$ and $n_j\geq8$, one has
\begin{align*}
&|f(z)|\\
\leq&\bigg[\ld^{m^{*}}\cdot \bigg(\frac{1}{2}\bigg)^{k}\cdot R_k^{n_k}\cdot\bigg(\prod^{k-1}_{j=1} R^{-n_j}_j\bigg)\bigg]\cdot \bigg[4\cdot R_k^{s_k}\cdot4^{m+\sum^{k}_{j=1}n_j}\bigg]\cdot2\\
=&\ld^{m^{*}}\cdot4\cdot \bigg(\frac{1}{2}\bigg)^{k-1}\cdot R_k^{m+\sum^{k-1}_{j=1}n_j}\cdot\bigg[\prod^{k-1}_{j=1} R^{-n_j}_j\bigg]\cdot 4^{m+\sum^{k}_{j=1}n_j},\\
\leq&\frac{1}{4}\cdot \ld^{m^{*}}\cdot 2^{m-1+\sum^{k+1}_{j=1}(n_{j}-2)} \cdot  R_{k+1}^{m+\sum^{k}_{j=1}n_{j}}\cdot\bigg[\prod^{k}_{j=1}{R_j}^{-n_j}\bigg]\\
\leq&\frac{1}{4}R_{k+2},
\end{align*}
where \eqref{inequ-1} and \eqref{valueest-8} are used.

Hence, one has $f(\partial_oA_k)\subset B_{k+1}$.

We will show $f(\partial_iA_k)\subset B_{k}$, that is, $$4R_{k}\leq\min_{z\in\partial_iA_k}|f(z)|\leq\max_{z\in\partial_iA_k}|f(z)|\leq\tfrac{1}{4}R_{k+1}.$$

First, we prove $\min_{z\in\partial_iA_k}|f(z)|\geq4R_{k}$.

By $m>10$ and $R>4$, one has
\beqq
R_{k-1}\geq4=2^2\ \text{and}\ R^{(m-2)+\sum^{k-1}_{j=1}n_j}_{k-1}\geq 2^{m+\sum^{k-1}_{j=1}(n_j+1)+6},
\eeqq
 or
\beqq
R^{(m-1)+\sum^{k-1}_{j=1}n_j}_{k-1}
\geq 2^{m+\sum^{k-1}_{j=1}(n_j+1)+5}\cdot (2R_{k-1}).
\eeqq
So, by \eqref{inequ-1},
\beqq
(R_k/(4R_{k-1}))^{(m-1)+\sum^{k-1}_{j=1}n_j}\geq R^{(m-1)+\sum^{k-1}_{j=1}n_j}_{k-1}
\geq 2^{m+\sum^{k-1}_{j=1}(n_j+1)+5}\cdot (2R_{k-1}),
\eeqq
yielding that
\beq\label{valueest-16}
 \bigg(\frac{1}{2}\bigg)^{2m+(\sum^{k-1}_{j=1}(2n_j+1))+2}\cdot R_k^{m+\sum^{k-1}_{j=1}n_j}\geq2^3\cdot(2R_{k-1})^{m+\sum^{k-1}_{j=1}n_j}.
\eeq
By \eqref{valueest-4}, \eqref{valueest-5}, and \eqref{valueest-13} ($\tau=\tfrac{1}{4}$ in Lemma \ref{valueest-6-8-26-2}), one has
\begin{align*}
&|f(z)|\\
\geq&\bigg[  \ld^{m^{*}}\cdot \bigg(\frac{1}{2}\bigg)^{k}\cdot R_k^{n_k}\cdot\bigg(\prod^{k-1}_{j=1} R^{-n_j}_j\bigg)\bigg]\cdot\bigg[ R_k^{s_k}\cdot\bigg(\frac{1}{4}\bigg)^{m+\sum^{k-1}_{j=1}n_j}\bigg]\cdot\frac{1}{2}\\
=&  \ld^{m^{*}}\cdot \bigg(\frac{1}{2}\bigg)^{2m+(\sum^{k-1}_{j=1}(2n_j+1))+2}\cdot R_k^{m+\sum^{k-1}_{j=1}n_j}\cdot\bigg[\prod^{k-1}_{j=1} R^{-n_j}_j\bigg] \\
\geq&4\cdot\frac{3}{2}\cdot\ld^{m^{*}}\cdot(2R_{k-1})^{m+\sum^{k-1}_{j=1}n_j}\cdot
\bigg[\prod^{k-1}_{j=1} R^{-n_j}_j\bigg]\geq 4R_{k},
\end{align*}
where \eqref{valueest-11} is used in the last inequality, and the last but one inequality is derived by \eqref{valueest-16}.

Now, we verify that $\max_{z\in\partial_iA_k}|f(z)|\leq\tfrac{1}{4}R_{k+1}.$

By definition of $R_{k+1}$, one has
\beqq
\max\{|f(z)|:\ |z|=R_k/4\}\leq\max\{|f(z)|:\ |z|=2R_k\}=R_{k+1}.
\eeqq
It follows from \eqref{valueest-4} and \eqref{valueest-5} that, where $R> 4$ is used for the factor $4$ in the second inequality, and \eqref{valueest-10} and
\eqref{valueest-15} (Lemmas \ref{valueest-6} and \ref{valueest-6-8-26-2}) are used in the following discussions,
\begin{align*}
&\frac{\max\{|f(z)|:\ |z|=R_k/4\}}{\max\{|f(z)|:\ |z|=2R_k\}}\\
\leq&
4\frac{\max\{|z^{s_k}\cdot(H_{n_k}(\frac{z}{R_k}))|:\ |z|=R_k/4\}}{\min\{|z^{s_k}\cdot(H_{n_k}(\frac{z}{R_k}))|:\ |z|=2R_k\}}\\
\leq&
4\frac{3\cdot R_k^{s_k}\cdot(\frac{1}{4})^{m+\sum^{k-1}_{j=1}n_j}}
{(\frac{1}{2})\cdot R_k^{s_k}\cdot2^{m+\sum^{k}_{j=1}n_j}}\\
\leq&\frac{1}{2^{2m-2+\sum^{k-1}_{j=1}3n_j+(n_k-3)}}\\
\leq &\frac{1}{4}.
\end{align*}
This gives us the required estimates.
\end{proof}

\begin{lemma}
Suppose (A*) and (A**) hold, one has $f(B_j)\subset B_{j+1}$,  $j\geq1$. As a consequence, $B_j$ is contained in the Fatou set of $f$, $j\geq1$.
\end{lemma}

\begin{proof}
By Lemma \ref{boundaryfatou}, the inner and outer boundary of $B_j$ is mapped into $B_{j+1}$. This, together with the fact that there is no zero $f$ in $B_j$ and the second part of Lemma \ref{valueest-12}, implies that $f(B_j)\subset B_{j+1}$.

Since $f(B_j)\subset B_{j+1}$ and any point $z\in B_j$, $\lim_{n\to\infty}f^n(z)\to\infty$ as $n\to+\infty$. Hence, the iterates of $f$ form a normal family on $B_j$. Therefore, $B_j$ is contained in the Fatou set of $f$.
\end{proof}

\subsection{The Julia set in $A_k$}

Recall that $\Om^p_m$ represents the petals of $\Om_m$, the $m$ components of $|H_m(z)|<1$ other than the central component $\Om^0_m$ that contains the origin.

\begin{lemma}\label{inequ-15}
Suppose (A*) and (A**) hold.  Then
 $\mathcal{J}(f)\cap A_j\subset V_j\cup(R_j\cdot\Om^{p}_{n_j})$, $j\geq1$.
\end{lemma}

\begin{proof}

The complement of $V_k\cup(R_k\cdot\Om^{p}_{n_k})$ in $A_k$ is divided into four pieces and we will verify that each of them is in the Fatou set.

First, consider the annulus $\{z:\ \tfrac{5R_k}{2}\leq|z|\leq4R_k\}$, where the boundary of this region consists of the outer boundary of $A_k$, $\{z:\ |z|=4R_k\}$ and the outer boundary of $V_k$, $\{z:\ |z|=\tfrac{5R_k}{2}\}$. By Lemmas \ref{valueest-17} and
\ref{boundaryfatou}, these two boundaries are mapped into $B_{k+1}$. By \eqref{valueest-18}, $f$ has no zeros in this annulus. So, the annulus is mapped into $B_{k+1}$, and this region is contained in the Fatou set by the second part of Lemma \ref{valueest-12}.

Second,  consider the region between the inner boundary of $A_k$, $\{z:\ |z|=\tfrac{R_k}{4}\}$, and the boundary of $R_k\cdot \Om^0_{n_k}$.

By Lemma \ref{inequ-17} and $n_k\geq8$, $\tfrac{1}{4}<1-\tfrac{1}{n_k}$ and the inner boundary of $A_k$ is contained in the interior of $R_k\cdot \Om^0_{n_k}$. The inner boundary of $A_k$ is mapped into $B_k$ by Lemma \ref{boundaryfatou}, and the inner boundary of $V_k$ is mapped into $B_k$ by Lemma \ref{valueest-17}, there is no zero of $f$ in this region. This, together with the minimum and maximum principles, yields that this region is mapped into $B_k$.

Third, consider the following region
\beq\label{inequ-16}
T^{\de}_k=\bigg\{z:\ 1-\frac{1}{n_k}\leq\frac{|z|}{R_k}\leq1+\frac{2}{n_k},\ |H_{n_k}(z/R_k)|>\de\bigg\},
\eeq
where $\de=\tfrac{4}{R^{m-3}_k}$ is a positive constant which is derived in the following discussions.
By \eqref{inequ-5} of Corollary \ref{inequ-6}, the petal region $R_k\cdot \Om^{p}_{n_k}$ is contained in this region
\beqq
\bigg\{z:\ 1-\frac{1}{n_k}\leq\frac{|z|}{R_k}\leq1+\frac{2}{n_k}\bigg\}.
\eeqq
So, $T^{\de}_k$ contains ``a large part" of the petal regions.

Next, the task is to show $f(T^{\de}_k)\subset B_k$, which can be derived by two inequalities:
\beqq
\max\{|f(z)|:\ z\in T^{\de}_k\}\leq\frac{1}{4}R_{k+1}\ \mbox{and}\ \min\{|f(z)|:\ z\in T^{\de}_k\}\geq 4R_k.
\eeqq

Note that $1+\tfrac{2}{n_k}\leq 1+\tfrac{2}{8}<\tfrac{3}{2}$ by $n_k\geq8$. This, together with the fact the inner boundary of $V_k$ is mapped into $B_k$ by Lemma \ref{valueest-17}, implies that the first inequality.

Now, we prove  $\min\{|f(z)|:\ z\in T^{\de}_k\}\geq 4R_k$.

Introduce a variable $a$ with $a\in[-1,2]$.  By \eqref{valueest-4} and \eqref{valueest-5}, one has
\begin{align}\label{inequ-9}
&\min\{|f(z)|:\ z\in T^{\de}_k\}\nonumber\\
\geq&
\frac{1}{2}
\ld^{m^{*}}\cdot \bigg(\frac{1}{2}\bigg)^{k}\cdot R_k^{n_k}\cdot\bigg[\prod^{k-1}_{j=1} R^{-n_j}_j\bigg]\cdot \bigg(\bigg(1+\frac{a}{n_k}\bigg)R_k\bigg)^{s_k}\cdot\de\nonumber\\
=&\ld^{m^{*}}\cdot \bigg(\frac{1}{2}\bigg)^{k+1}\cdot R_k^{(m+\sum^{k-1}_{j=1}n_j)}\cdot\bigg[\prod^{k-1}_{j=1} R^{-n_j}_j\bigg]\cdot \bigg(1+\frac{a}{n_k}\bigg)^{s_{k}}\cdot\de\nonumber\\
=&\ld^{m^{*}}\cdot \bigg(\frac{1}{2}\bigg)^{k+1}\cdot R_k^{(m+\sum^{k-1}_{j=1}n_j)}\cdot\bigg[\prod^{k-1}_{j=1} R^{-n_j}_j\bigg]\cdot \bigg(1+\frac{a}{n_k}\bigg)^{m_{k-1}-n_k}\cdot\de.
\end{align}

Now, we estimate $\big(1+\frac{a}{n_k}\big)^{m_{k-1}-n_k}$ for $a\in[-1,2]$.
This is split into two parts: $\big(1+\frac{a}{n_k}\big)^{m_{k-1}}$ and $\big(1+\frac{a}{n_k}\big)^{-n_k}$.

For $a\in[-1,2]$, by $n_k\geq8$,
\beqq
\bigg(1+\frac{a}{n_k}\bigg)^{m_{k-1}}\geq\bigg(\frac{1}{2}\bigg)^{m_{k-1}}.
\eeqq
For $a\in[-1,2]$, one has $1+\tfrac{a}{n_k}\leq1+\tfrac{|a|}{n_k}$. So, it suffices to consider the case $a\in[0,2]$. Take a sufficiently small positive constant $\eta$ ($\eta<\tfrac{1}{2}$), and the interval $[0,2]$ is split into two parts: $[0,\eta]$ and $[\eta,2]$.

Now, we consider the case $a\in[\eta,2]$.

The classical Taylor expansion gives the following identity:
\beq\label{inequ-7}
x\bigg(\bigg(1+\frac{1}{x}\bigg)^x-e\bigg)=-\frac{e}{2}+\frac{11e}{24}\frac{1}{x}+O\bigg(\frac{1}{x^2}\bigg)\ \mbox{for real}\ x.
\eeq
By the assumption $n_k>4$, for $a\in[\eta,2]$, $\tfrac{n_k}{a}\geq\tfrac{n_k}{2}\geq2$. So, by \eqref{inequ-7}, it suffices to use the inequality $\tfrac{e}{2}\leq(1+\tfrac{1}{x})^x\leq2e$ for $x\geq2$. So, one has
\begin{align*}
 \bigg(1+\frac{a}{n_k}\bigg)^{n_k}=\bigg(\bigg(1+\frac{a}{n_k}\bigg)^{\frac{n_k}{a}}\bigg)^{a}
\leq(2e)^{a}<6^2=36.
\end{align*}
Now, we study the situation $a\in[0,\eta]$.
\beqq
\bigg(1+\frac{a}{n_k}\bigg)^{n_k}\leq 1+2\frac{a}{n_k}n_k\leq1+2\eta<2.
\eeqq
Hence, by \eqref{inequ-9}, one has
\begin{align*}
&\min\{|f(z)|:\ z\in T^{\de}_k\}\\
\geq&
 \ld^{m^{*}}\cdot \bigg(\frac{1}{2}\bigg)^{k+1}\cdot R_k^{(m+\sum^{k-1}_{j=1}n_j)}\cdot\bigg[\prod^{k-1}_{j=1} R^{-n_j}_j\bigg]\cdot \bigg(\frac{1}{(1+\tfrac{a}{n_k})}\bigg)^{n_k}\cdot
 \bigg(\frac{1}{2}\bigg)^{m+\sum^{k-1}_{j=1}n_j}\cdot\de\\
\geq&  \ld^{m^{*}}\bigg(\frac{1}{2}\bigg)^{m+1}\cdot R^{m}_k\cdot \bigg[\prod^{k-1}_{j=1} \bigg(\frac{R_k}{4R_j}\bigg)^{n_j}\bigg]\cdot \frac{1}{36}\cdot \de\\
\geq& R^{m-2}_k\cdot \de\geq 4R_k,
\end{align*}
where the last inequality can be derived by  $\ld^{m^{*}}=(2\ld)^{m-1}$, \eqref{inequ-8}, $m\geq2^4$, $R\geq2^5$, and
\beq\label{est-7-31-2}
\de\geq\frac{4}{R^{m-3}_k}.
\eeq
This constant $\de$ can be arbitrarily small as we want as long as $R$ is sufficiently large.

Fourth, consider the following regions, which are not the union of the above three regions and outside the region $V_k$:
\beqq
\bigg\{z:\ \frac{1}{4}\leq\frac{|z|}{R_k}\leq1-\frac{1}{n_k},\ z\not\in R_k\cdot \Om^0_{n_k}\bigg\}
\eeqq
and
\beqq
\bigg\{z:\ 1+\frac{2}{n_k}\leq\frac{|z|}{R_k}\leq\frac{3}{2}\bigg\}.
\eeqq
It is sufficient to show that for $z\in \big\{z:\ \tfrac{|z|}{R_k}=1-\tfrac{1}{n_k}\big\}\cup\big\{z:\ \tfrac{|z|}{R_k}=1+\tfrac{2}{n_k}\big\}$, $|H_{n_k}(\tfrac{z}{R_k})|>\de$, where $\de$ is specified in \eqref{est-7-31-2}.

Now, show the function $g_1(x)=(1+\tfrac{1}{x})^x$ with $x\geq2$ is nondecreasing. Direct computation gives the derivative of $g_1(x)$ is $ (1+\tfrac{1}{x})^x\cdot(\log(1+\tfrac{1}{x})-\tfrac{1}{x+1})$. For the function $g_2(x)=\log(1+\tfrac{1}{x})-\tfrac{1}{x+1}$, $g_2(1)=\log(2)-0.5\approx0.693147-0.5=0.193147>0$, $\lim_{x\to+\infty}\log(1+\tfrac{1}{x})=\lim_{x\to+\infty}\tfrac{1}{x+1}=0$, the derivative of $g_2(x)$ is $-\tfrac{1}{x^2+x}+\tfrac{1}{(x+1)^2}<0$. So, the derivative of $g_1(x)$ is nonnegative. Hence, $g_1(x)=(1+\tfrac{1}{x})^x$ with $x\geq2$ is nondecreasing. Similarly, the function $g_3(x)=(1-\tfrac{1}{x})^x$ with $x\geq2$ is nondecreasing, since $g_3(x)=(1-\tfrac{1}{x})^{(-x)\cdot(-1)}=\tfrac{1}{g_1(-x)}$.

Direct computation gives, for $z\in \big\{z:\ \tfrac{|z|}{R_k}=1-\tfrac{1}{n_k}\big\}$, one has
\beqq
|H_{n_k}(\tfrac{z}{R_k})|\geq\big(1-\tfrac{1}{n_k}\big)^{n_k}(2-\big(1-\tfrac{1}{n_k}\big)^{n_k})
\geq(1-\tfrac{1}{8})^8\cdot(2-\tfrac{1}{e})\approx0.560811;
\eeqq
for $z\in \big\{z:\ \tfrac{|z|}{R_k}=1+\tfrac{2}{n_k}\big\}$, one has
\beqq
|H_{n_k}(\tfrac{z}{R_k})|\geq\big(1+\tfrac{2}{n_k}\big)^{n_k}(\big(1+\tfrac{2}{n_k}\big)^{n_k}-2)
\geq(1+\tfrac{2}{8})^8\cdot(\big(1+\tfrac{2}{8}\big)^{8}-2)\approx23.6062.
\eeqq
So, for sufficiently large $R$, one can take $\frac{4}{R^{m-3}_1}<0.5$.

The inner boundary of $A_k$ is mapped into $B_k$ by Lemma \ref{boundaryfatou}, the inner boundary of $V_k$ is mapped into $B_k$ by Lemma \ref{valueest-17}, and there is no zero of $f$ in this region. This, together with the minimum and maximum principles, implies that these two regions are mapped into $B_k$.

\end{proof}

\begin{remark}\label{dis8-2-2}
By Definition \ref{petal731-1}, $H_{n_k}$ is a conformal map of each petal in $\Om^p_{n_k}$ to the unit disk, implying that each part of the petal where $|H_{n_k}|\leq\de$ has diameter similar to the multiplication of $\de$ and the diameter of the petal.
Recall the definition of $A_k$ in \eqref{equ2021-1-26-1}, by Corollary \ref{inequ-6}, the diameter of the components of $R_k\cdot\Om^p_{n_k}$ is about $O(\tfrac{R_k}{n_k})$. This, together with \eqref{est-7-31-2}, implies that the part of the Julia set contained in each petal has diameter at most
\beqq
\frac{R_k}{n_k}\cdot \frac{4}{R^{m-3}_k}= \frac{4}{n_k\cdot R^{m-4}_k}.
\eeqq

This estimate, together with the generalized Koebe distortion estimate (Lemma \ref{distorsion-7-31-3}), will be applied in the dimension estimation.
\end{remark}

\subsection{Critical points in the Fatou set}

In this subsection, the critical points of $f$ are verified to be in the Fatou set.

An entire function is hyperbolic, if the set of singular values, including critical values and finite asymptotic values, is bounded and all such points iterate to attracting cycles
\cite{RempeSixsmith2017}. Although the functions considered here have an unbounded set of critical values, implying that these are not hyperbolic, all the critical points of these functions are in the Fatou set.

\begin{lemma}\label{criticalfatou-8-17-3}
Suppose (A*) and (A**) hold.
For any critical point of $f$ in $A_k$, the image of this point is in $B_k$, implying that the critical point is in the Fatou set.
\end{lemma}

\begin{proof}
By Lemma \ref{criticalvk-8-6-3}, there is no critical point in $V_k$. This, together with Lemma \ref{inequ-15}, yields that one needs to show that there is no critical point in the petals.

It follows from \eqref{valueest-4} and \eqref{valueest-5} that the function $f$ can be rewritten as
\beqq
f(z)=C_k\cdot z^{s_k}\cdot\bigg(H_{n_k}\bigg(\frac{z}{R_k}\bigg)\bigg)\cdot(1+h(z)))
\eeqq
where $s_k$ is introduced in \eqref{exp8-20-1}, $h(z)$ is holomorphic on $A_k$, and $|h(z)|=O(R^{-1}_k)$ on $A_k$. This, together with the Cauchy's estimate, implies that $|h^{\prime}(z)|=O(R^{-2}_k)$ for $z\in A_k$.

Taking derivative of $f(z)$, one has
\begin{align}\label{deriva-1}
f^{\prime}(z)&=C_k\cdot s_k\cdot z^{s_k-1}\cdot\bigg(H_{n_k}\bigg(\frac{z}{R_k}\bigg)\bigg)\cdot(1+h(z))\nonumber\\
&+C_k\cdot z^{s_k}\cdot H^{\prime}_{n_k}\bigg(\frac{z}{R_k}\bigg)\cdot\frac{1}{R_k}\cdot(1+h(z))\nonumber\\
&+C_k\cdot z^{s_k}\cdot\bigg(H_{n_k}\bigg(\frac{z}{R_k}\bigg)\bigg)\cdot h^{\prime}(z).
\end{align}
The critical points are solutions to the equations
\beq\label{deriva-2}
z^{s_k-1}=0\
\eeq
and
\begin{align}\label{deriva-3}
 s_k\cdot H_{n_k}\bigg(\frac{z}{R_k}\bigg)\cdot(1+h(z))+ z\cdot H^{\prime}_{n_k}\bigg(\frac{z}{R_k}\bigg)\cdot\frac{1}{R_k}\cdot(1+h(z))+ z\cdot H_{n_k}\bigg(\frac{z}{R_k}\bigg)\cdot h^{\prime}(z)=0.
\end{align}
Since we consider the critical points in $A_k$, $z^{s_k-1}=0$ is impossible.

Now, we consider the second equation.
\begin{align*}
 H^{\prime}_{n_k}\bigg(\frac{z}{R_k}\bigg)&=
R_k\cdot H_{n_k}\bigg(\frac{z}{R_k}\bigg)\cdot\bigg[\frac{-s_k\cdot(1+h(z))- z\cdot h^{\prime}(z)}{z\cdot(1+h(z))}\bigg]\\
&=R_k\cdot H_{n_k}\bigg(\frac{z}{R_k}\bigg)\cdot\bigg(-\frac{s_k}{z}-\frac{h^{\prime}(z)}{1+h(z)}\bigg)\\
&=R_k\cdot H_{n_k}\bigg(\frac{z}{R_k}\bigg)\cdot\bigg(-\frac{s_k}{z}-\frac{O(R^{-2}_k)}{1+O(R^{-1}_k)}\bigg)\\
&=R_k\cdot H_{n_k}\bigg(\frac{z}{R_k}\bigg)\cdot [s_k\cdot O(R^{-1}_k)+O(R^{-2}_k)].
\end{align*}
This, together with the fact
\beqq
\max_{|w|\leq2}|H_{n_k}(w)|\leq 2^{n_k}(2+2^{n_k})\leq2\cdot2^{2n_k}=2^{1+2n_k},
\eeqq
yields that
\beqq
H^{\prime}_{n_k}\bigg(\frac{z}{R_k}\bigg)=\left\{
  \begin{array}{ll}
   O\big(\frac{2^{1+2n_k}}{R_k}\big) , & \hbox{if}\ s_k=0 \\
  O\big(s_k\cdot2^{1+2n_k}\big)  , & \hbox{if}\ s_k\neq0.
  \end{array}
\right.
\eeqq
On the other hand,
\beqq
1-H_{n_k}(z)=1-z^{n_k}(2-z^{n_k})=(1-z^{n_k})^2
=\bigg(\frac{H^{\prime}_{n_k}(z)}{2\cdot n_k\cdot z^{n_k-1}}\bigg)^2.
\eeqq
So, at a critical point of $f$,
\begin{align}\label{est-9}
&|1-H_{n_k}(z/R_k)|
=\bigg|\frac{H^{\prime}_{n_k}(z/R_k)}{2\cdot n_k\cdot z^{n_k-1}}\bigg|^2
\leq \bigg|\frac{H^{\prime}_{n_k}(z/R_k)}{2n_k}\cdot\bigg(\frac{2}{R_k}\bigg)^{n_k-1}\bigg|^2\nonumber\\
\leq&\left\{
  \begin{array}{ll}
   \frac{2^{6n_k}}{4\cdot n^2_k\cdot R^{2n_k}_k} , & \hbox{if}\ s_k=0 \\
&\\
  \frac{s^2_k\cdot2^{6n_k}}{4\cdot n^2_k\cdot R^{2(n_k-1)}_k}, & \hbox{if}\ s_k\neq0.
  \end{array}
\right.
\end{align}
By \eqref{inequ-20}, one has
\begin{align*}
\frac{s^2_k}{n^2_k}\leq\frac{2\cdot n_k^2+2(m+\sum^{k-1}_{j=1}n_j)^2}{n^2_k}
\leq\frac{2\cdot n_k^2+2(2\cdot\tfrac{\log R_k}{\log 2})^2}{n^2_k}
\leq2+\frac{8(\log R_k)^2}{n^2_k},
\end{align*}
so,
\begin{align*}
 \frac{s^2_k\cdot2^{6n_k}}{4\cdot n^2_k\cdot R^{2(n_k-1)}_k}\leq
\frac{2^{6n_k}}{2\cdot R^{2(n_k-1)}_k}+\frac{2^{6n_k+1}\cdot(\log R_k)^2}{n^2_k\cdot R^{2(n_k-1)}_k}.
\end{align*}
By $n_k>4$ and $R_k\geq2^6$ (this is derived by $R\geq2^5$), one has $\tfrac{2^{6n_k}}{R^{2(n_k-1)}}\leq\tfrac{1}{2}$.
 By combing the above arguments, one has
\beqq
|1-H_{n_k}(z/R_k)|\leq\frac{1}{4}.
\eeqq
Hence, at such a critical point, one has $|H_{n_k}(z/R_k)|\geq\tfrac{3}{4}$.
This, together with the discussions in the proof of Lemma \ref{inequ-15}, implies the conclusion of this lemma.
\end{proof}

\begin{lemma}\label{equ-2021-2-13-3}
The constants $R$ in the construction of the function $f$ can be taken large enough such that the critical points of $f$ in $\{z:\in\mathbb{C}:\ |z|<R\}$ are in the Fatou set. Furthermore, the constant $R$ can be chosen as large as we wish.
\end{lemma}

\begin{proof}
This lemma can be derived by the same arguments as in the proof of Lemma 14.3 in \cite{Bishop2018}.
\end{proof}

\subsection{Negative indices}\label{negative8-2-3}
In this subsection, we define the sets $A_k$, $V_k$, and $U_k$ for $k\leq0$.

For any point $z$, the forward orbit of $z$ is denoted by $\mbox{Orb}(z)=\{f^n(z):\ n\in\mathbb{N}\}$. If $\mbox{Orb}(z)\cap(\cup^{\infty}_{k=1}A_k)$ is an infinite set, then there may exist positive integers $k_0$ and $l_0$ such that $f^{k_0}(z)$ is in a small neighborhood of the origin, $\{f^{k_0+1}(z),\ f^{k_0+2}(z),....,f^{k_0+l_0}(z)\}\subset D_1$, and $f^{k_0+l_0+1}(z)\in A_1$.

Set
$$A_0:=\{z\in D_1:\ f(z)\in A_1\},$$
\beqq
A_{-k}:=\bigg\{z\in D_1:\ \{z,f(z),...,f^k(z)\}\subset D_1\ \mbox{and}\ f^{k+1}(z)\in A_1\bigg\},\ k\geq1.
\eeqq
Similarly, we could define $V_k$ and $U_k$ for $k\leq0$.
Denote
$$V_0:=\{z\in D_1:\ f(z)\in V_1\},$$
\beqq
V_{-k}:=\bigg\{z\in D_1:\ \{z,f(z),...,f^k(z)\}\subset D_1\ \mbox{and}\ f^{k+1}(z)\in V_1\bigg\},\ k\geq1;
\eeqq
and
$$U_0:=\{z\in D_1:\ f(z)\in U_1\},$$
\beqq
U_{-k}:=\bigg\{z\in D_1:\ \{z,f(z),...,f^k(z)\}\subset D_1\ \mbox{and}\ f^{k+1}(z)\in U_1\bigg\},\ k\geq1.
\eeqq

Now, we study the properties of the set $V_{-k}$ for $k\geq0$.

By Remark \ref{equ2021-7-29-1}, the disk $D(0,R)$ contains $2^N-1$ critical values of $f$, where these critical values are in the same Fatou component.  So, there exists a positive integer $T$ such that $V_{-k}$ surrounds all these $2^N-1$ critical values for $k=0,1,...,T$, and $V_{-k}$ does not surround any critical point for $k>T$. So, there is only one connected component of $V_{-k}$ for $k=0,1,...,T$, there are $2^{jN}$ connected components for $V_{-T-j}$ for any $j\geq1$. Hence, $f$ is a $2^N$-to-$1$ covering map from $V_{-k}$ to $V_{-k+1}$ for $k=0,...,T$, and each connected component of $V_{-T-j}$ has $2^N$ distinct connected components under the pre-image of $f$ for $j\geq1$. Hence, $V_{-k}$ is a union of topological annuli that surrounds the Cantor set $E$ for any $k\geq0$, and each component of $V_{-k}$ is mapped to a component of $V_{-k+1}$.

Recall the definition of $m_k$ ($k\geq1$) in \eqref{est-8}, the indices for $m_k$ ($k\leq0$) are defined:
\beq\label{est-12}
m_k:=\left\{
  \begin{array}{ll}
   2^N, & \hbox{for}\ -T\leq k\leq 0 \\
    1, & \hbox{for}\ k<-T.
  \end{array}
\right.
\eeq
The covering map $f:A_{-k}\to A_{-k+1}$ for $k\geq0$ has degree $m_k$. Let $M_0=2^{NT}=\prod_{k\leq0}m_k$, this gives an upper bound of the pre-images of a single point $z\in V_1$ that will be discovered in any connected component of $V_{-k}$, $k\geq0$.

\subsection{Partitioning the Julia set}

In this subsection, the Julia set is split into two parts according to the orbits of the points in the Julia set. For an illustration diagram of the Julia set, please refer to Figure 6 in \cite{Bishop2018}.

\begin{lemma}\label{fatou8-2-1}
\begin{itemize}
\item [(i)]Any connected component $W$ of $f^{-1}(A_j)$ is contained in $A_k$ for some $k\geq j-1$, where $j\in\mathbb{Z}$.
    \item[(ii)] The connected components of $f^{-1}(A_j)$ contained in $A_k$, $k\geq j$, are inside the petals $R_k\cdot \Om^{p}_{n_{k}}$, where $j\geq1$.
        \end{itemize}
\end{lemma}

\begin{proof}
{\bf Case (i)}
If $j\leq0$, then $f(A_j)=A_{j+1}$. If $j\geq1$, by  \eqref{valueest-4} and \eqref{valueest-5}, $f(A_k)\cap A_j=\emptyset$ for $k<j-1$. So,
any connected component $W$ of $f^{-1}(A_j)$ is contained in $A_k$ for some $k\geq j-1$.

{\bf Case (ii)} The arguments in the proof of Lemma \ref{inequ-15} will be used here.

The region considered in \eqref{inequ-16} contains the boundary of $R_k\cdot\Om^{\infty}_{n_k}$ by Lemma \ref{inequ-18}, by the conclusions there and $|H_{n_k}|\geq1$ on $\Om^{\infty}_{n_k}$, we know that there are no pre-images in  $R_k\cdot\Om^{\infty}_{n_k}$.

Now, we consider the region in $A_k\cap(R_k\cdot\Om^0_{n_k})$. By Lemma \ref{inequ-17}, the boundary of $R_k\cdot\Om^0_{n_k}$ is contained in the region defined in \eqref{inequ-16}. Since the inner boundary of $A_k$ is mapped into $B_k$, and the boundary of $R_k\cdot\Om^0_{n_k}$ is contained in \eqref{inequ-16}, where $|H_{n_k}|=1$ on the boundary of $\Om^0_{n_k}$. This, together with the fact that $f$ has no zeros in $A_k\cap(R_k\cdot\Om^0_{n_k})$ and the minimum principle, implies the conclusion of this lemma.
\end{proof}

\subsubsection{Julia set of small dimension}

\begin{definition}
For a bounded domain $G$ in $\mathbb{C}$, let $U(G)$ be the unbounded component of $\mathbb{C}\setminus G$.
 The set $\widehat{G} = \mathbb{C}\setminus U(G)$ is said to be the topological hull of $G$. Thus $\widehat{G}$ is the union of $G$ and
the bounded components of its complement. Informally, $\widehat{G}$ is obtained from $G$ by
``filling in the holes" of $G$.
\end{definition}

\begin{lemma}\cite[Theorem 2.9]{McMullen1994}\label{distorsion-7-31-3}
Let $D\subset U\subset\mathbb{C}$ be disks with
$\mbox{mod}(D,U)>m>0$. Let $f:U\to\mathbb{C}$ be a univalent map. Then there is a constant $C(m)$ such that for any $x$, $y$ and $z$ in $D$,
\beqq
\frac{1}{C(m)}|f'(x)|\leq\frac{|f(y)-f(z)|}{|y-z|}\leq C(m)|f'(x)|.
\eeqq
\end{lemma}

The idea of the arguments for the following lemma is similar with the one used in the proof of Lemma 16.3 in \cite{Bishop2018}, where some estimates are different here.

\begin{lemma}\label{smalldim-1}
 Let $Y\subset X$ be the set of points $z$ satisfying $k(z,n+1)\leq k(z,n)$ infinitely often. Given any small positive constant $\al$, if the constants $\ld$, $R$, and $N$ are sufficiently large, then $\mbox{dim}(Y)\leq \al$
\end{lemma}

\begin{proof}
The arguments follows the idea of the proof of Lemma 16.3 of \cite{Bishop2018}. The idea of the construction of the covers is the same, but the estimates of diameter of the elements in covers are different.

The set $Y$ is a subset of the union of $A_k$, i.e., $Y\subset \cup A_k$, and $Y$ is an invariant subset of $f$.

Now, we show that $\mbox{dim}(Y\cap A_m)\leq\al$ for any $m\geq1$. The idea of the proof is the introduction of nested covering of $Y\cap A_m$.

For any $k\geq m$, let $W^n_k$ be some components of $f^{-n}(A_k)$ that lie inside $A_m$.
 For convenience, $W^0_m=f^{0}(A_m)=A_m$ is the first covering of $Y\cap A_m$. For $z\in W^0_m\cap Y$, by the definition of $Y$, $k(z,n+1)\leq k(z,n)$ infinitely often, it is possible that $f(z)\in A_{m_0}$ with $m_0\leq m$ or $f(z)\in A_{m+1}$. If $f(z)\in A_{m_0}$, then we stop and cover $z$ by a component of $f^{-1}(A_{m_0})$; if $f(z)\in A_{m+1}$, we continue the iteration and wait for the minimal positive integer $q$ such that $f^{q-1}(z)\in A_{m+q-1}$ and $f^{q}(z)\in A_{m'_0}$ with $m'_0\leq m+q-1$, then we stop and cover $z$ by a component of $f^{-q}(A_{m'_0})$.

Inductively, a refinement of a sequence of nested covers for $Y\cap A_m$ is defined as follows.  Suppose
$W^n_k\subset f^{-n}(A_k)$ is an element of the current cover, for $z\in W^n_k\cap Y$, it follows from the definition of $Y$ that $f^{n+q}(z)\in A_j$ with $j\leq k+q-1$, where $q$ is the minimal positive integer. At this point we stop and cover $z$ by a component of the form $W^{n+q}_j$, that is, some component of $f^{-(n+q)}(A_j)$. Thus, $Y\cap W^n_k$ can be covered by components of the form $W^{n+q}_j$, where $q\geq1$ and $j\leq k+q-1$. So, the set $W^n_k\cap Y$ can be covered by a refinement covering, where the components defined in this way.

For the cover for $Y$,  $W^n_k$ can be replaced by the topological hull $\widehat{W}^n_k$, since every component with $j<k+q-1$ is contained in some hole of a topological hull, and $W^{n+q}_{k+q-1}$ and $\widehat{W}^{n+q}_{k+q-1}$ have the same diameter, implying that it is sufficient to consider the case $j=k+q-1$,  Thus, the application of these covers (topological hull) does not change the sum in the definition of Hausdorff measure and dimension. Note that using the filled-in components requires us to consider the cases $q=0$ and $q\geq1$, where $q=0$ and $q\geq1$ correspond to  the first and second cases of Lemma \ref{fatou8-2-1}, respectively.

If the $\al$-sum of the refinement of the covers tends to zero, then the $\mbox{dim}(Y\cap A_m)\leq\al$. The decay rate for the $\al$-sum is geometrically fast. There are two different situations ($q=0$ and $q>0$):
\beq\label{equest-1}
\sum_{W^{n}_{k-1}\subset\widehat{W}^n_{k}}\mbox{diam}(W^n_{k-1})^{\al}\leq\frac{1}{4}\mbox{diam}(W^n_k)^{\al}\ \mbox{for}\ q=0
\eeq
and
\beq\label{inequ-24}
\sum_{q\geq1}\sum_{W^{n+q}_{k+q-1}\subset W^n_k}\mbox{diam}(W^{n+q}_{k+q-1})^{\al}\leq\frac{1}{4}\mbox{diam}(W^n_k)^{\al}\ \mbox{for}\ q>0,
\eeq
where the case $q=0$ refers that $Y\cap(\widehat{W}^n_k\setminus W^n_k)$ is covered by the components of the form $\widehat{W}^n_{k-1}$ in the next generation.

We show \eqref{equest-1}.

For $k\geq1$, it follows from Lemma \ref{distorsion-7-31-3} and Remark \ref{dis8-2-2} that
\beqq
\frac{\mbox{diam}(W^n_{k-1})}{\mbox{diam}(W^n_k)}
\leq C\frac{\mbox{diam}(f^n(W^n_{k-1}))}{\mbox{diam}(f^n(W^n_k))}
=C\frac{\mbox{diam}(A_{k-1})}{\mbox{diam}(A_k)}\leq C\frac{1}{R_1},
\eeqq
where $C$ is a positive constant determined by Remark \ref{dis8-2-2} and Lemma \ref{distorsion-7-31-3}, and it is independent on $k$.

So, $W^n_{k-1}$ has one component in $\widehat{W}^n_k$ and the diameter is $O(R^{-1}_1)\cdot \mbox{diam}(W^n_k)$. For $k\leq0$, there is a bounded number of connected components of $W^{n}_{k-1}$ inside $W^n_k$, where this number is dependent on the number of $N$. By Lemma \ref{equ2021-2-13-2}, Remark \ref{equ2021-7-29-1}, and the discussions in Subsection \ref{negative8-2-3}, the ratio of the diameter of each component of $W^n_{k-1}$ and the diameter of $W^n_{k}$ is small for large enough $\ld$. Hence,
\eqref{equest-1} can be derived by these arguments.

Now, we prove \eqref{inequ-24}.

By the refinement of the covers, one has
\begin{center}
\begin{align*}
W^{n+q}_{k+q-1}\subset W^n_k\subset& A_m,\\
f^n(W^{n+q}_{k+q-1})\subset f^n(W^n_k)=&A_k,\\
f^{n+1}(W^{n+q}_{k+q-1})\subset& A_{k+1},\\
f^{n+2}(W^{n+q}_{k+q-1})\subset& A_{k+2},\\
\vdots\quad\quad\vdots &\\
f^{n+q-1}(W^{n+q}_{k+q-1})\subset& A_{k+q-1},\\
f^{n+q}(W^{n+q}_{k+q-1})\subset& A_{k+q-1}.
\end{align*}
\end{center}

Recall the definition of $m_k$ in \eqref{est-8} ($k>0$) and \eqref{est-12} ($k\leq0$), and \eqref{inequ-19} and \eqref{inequ-20} give upper bounds for $m_k$ ($k>0$).

The first $q-1$ maps are restrictions of the covering maps $A_{k+i-1}\supset f^{-1}(A_{k+i})\to A_{k+i}$, and the final one is the restriction of a petal. The $i$th covering map for $i=1,...,q-1$ is
\beq\label{degreeest-8-4-3}
\left\{
  \begin{array}{ll}
   m_{k+i-1}\text{-}\mbox{to}\text{-}1,\ & \hbox{if}\ k+i-1\geq1\ \ (\mbox{by the degree defined in}\ \eqref{est-8})\\
    2^N\text{-}\mbox{to}\text{-}1,\ & \hbox{if}\ -T\leq k+i-1\leq0\ \ (\mbox{by the discussions in Subsection \ref{negative8-2-3}})\\
    1\text{-}\mbox{to}\text{-}1,\ & \hbox{if}\ k\leq-T\ \ (\mbox{by the discussions in Subsection \ref{negative8-2-3}}).
  \end{array}
\right.
\eeq
The number of possible new components bringing by the refinement of the cover has an upper bound:
\beqq
2^{NT}\cdot m_k\cdot m_{k+1}\cdot m_{k+2}\cdots m_{k+q-2}\ \forall k\geq1.
\eeqq
The size of a single pre-image is given by the final petal map:
\beqq
\mbox{diam}(W^{n+q}_{k+q-1})\leq \frac{R_{k+q-1}}{R_{k+q}}\cdot\mbox{diam}(W^n_k)\leq\frac{\mbox{diam}(W^n_k)}{R_{k+q-1}}.
\eeqq
So, by \eqref{inequ-19} and \eqref{inequ-20} in Corollary \ref{degreeest-8-4-1}, one has
\begin{align*}
& 2^{NT}\cdot m_k\cdot m_{k+1}\cdot m_{k+2}\cdots m_{k+q-3}\cdot m_{k+q-2}\cdot\bigg(\frac{\mbox{diam}(W^n_k)}{R_{k+q-1}}\bigg)^{\alpha}\\
\leq&   2^{NT}\cdot \bigg(2\frac{\log R_{k+2}}{\log R_{k+1}}\bigg)\cdot
\bigg(2\frac{\log R_{k+3}}{\log R_{k+2}}\bigg)\cdot\bigg(2\frac{\log R_{k+4}}{\log R_{k+3}}\bigg)\cdots\\
&\times \bigg(2\frac{\log R_{k+q-1}}{\log R_{k+q-2}}\bigg)\cdot
\bigg(2\frac{\log R_{k+q-1}}{\log 2}\bigg) \cdot \bigg(\frac{\mbox{diam}(W^n_k)}{R_{k+q-1}}\bigg)^{\alpha}\\
=&  \frac{2^{NT} 2^{q}}{\log2\log R_{k+1}}\cdot \frac{(\log R_{k+q-1})^2}{R^{\alpha}_{k+q-1}}\cdot\mbox{diam}(W^n_k)^{\alpha}.
\end{align*}
This, together with Lemma \ref{seriescon-8-4-2}, yields \eqref{inequ-24}.
\end{proof}

\subsubsection{Julia set in the escaping set}
In this subsection, a geometric description of the set $Z$ is given.

\begin{lemma}\cite[Lemma 18.1]{Bishop2018}\label{circleclose}
Suppose $h$ is a holomorphic function on $A=\{z:\ 1<|z|<4\}$ and $|h|$ is bounded by $\varepsilon$ on $A$. Let $H(z)=(1+h(z))z^l$, where $l$ is a non-zero integer. For any fixed $\tht$, the segment $S(\tht)=\{r e^{i\tht}:\ \tfrac{3}{2}\leq r\leq\tfrac{5}{2}\}$ is mapped by $H$ to a curve that makes angle at most $O(\tfrac{\varepsilon}{l})$ with any radial ray it meets.
\end{lemma}

\begin{remark}\label{roundann-8-17-1}
Let $W\subset V_k$ be the pre-image of $V_{k+1}$ under the map $f$. Note that $f$ is a small perturbation of a power function restricted to $V_{k+1}$ by \eqref{inequ-21} of Lemma \ref{est-1}. This, together with the fact that the component of the pre-image of a round annulus under a power function is another round annulus, implies that $W$ is a small perturbation of a round annulus. Lemma \ref{circleclose} gives a precise description of this fact.
\end{remark}

\begin{lemma}\label{packone}
Let $Z\subset X$ be the set of points $z$ in the Julia set with $k(z,n+1)=k(z,n)+1$ for all sufficiently large $n$. Then $Z$ is a union of $C^1$ closed Jordan curves, and $Z$ has locally finite $1$-measure.
\end{lemma}

\begin{proof}
The idea of the arguments follows from the study of Lemma 16.4 of \cite{Bishop2018}. For convenience of the readers, an outline of the whole arguments is provided.

For $W$ introduced in Remark \ref{roundann-8-17-1}, the width is approximately $R_k/m_k$, and each boundary component of $W$ is a smooth curve which is $\ep_k$-close to circles by Lemma \ref{circleclose}, where $\ep_k$ is specified in \eqref{appep-8-17-2}.

For $k,n\in\mathbb{N}$, consider the set
\beqq
\Gamma_{k,n}=\{z\in A_k:\ f^j(z)\in A_{k+j},\ j=1,...,n\}.
\eeqq
For fixed $k$, $\Ga_{k,n+1}\subset\Ga_{k,n}$, that is, these are nested topological annuli with widths decaying to zero uniformly.
The pulling back circles in $A_{k+n}$ by the map $f^{-n}$, which go around $\Gamma_{k,n}$ once, form a foliation of $\Gamma_{k,n}$. The angle between the foliation curves of $\Gamma_{k,n+1}$ and those of $\Gamma_{k,n}$ is at most $O(\ep_k)$ by Lemma \ref{circleclose}. By the expression of $\ep_k$ in \eqref{appep-8-17-2} with $l_k=1$, one has that $O(\sum_{k\geq1}\ep_k)$ is finite. This, together with $\Ga_{k,n+1}\subset\Ga_{k,n}$, yields that the limit of $\Ga_{k,n}$ as $n\to\infty$ is a $C^1$ Jordan curve, denoted by $\Ga_k$, which makes angle with the circular arcs foliating $V_k$ no larger than $O(\sum_{k\geq1}\ep_k)$. Hence, the length of $\Ga_k$ is a multiple of its diameter.

Now, to show the Hausdorff measure is finite, it suffices to show that the sums of the lengths of all the components of $Z$ in a bounded region of the plane is finite. Since each component of $Z$ is associated to a unique set of the form $W^n_k$, where $W^n_k$ is introduced in the proof of Lemma \ref{smalldim-1}, the sum of the diameters over components of $Z$ is dominated by the sum of diameters over sets of the form $W^n_k$ and $\al=1$ in the arguments of the proof of Lemma \ref{smalldim-1}, including \eqref{equest-1} and \eqref{inequ-24}. This, together with the fact that the exponentially decay rate of the estimate in \eqref{equest-1} and \eqref{inequ-24}, yields that the sum is finite.
\end{proof}

\subsection{The shape of the Fatou components}

In this subsection, the geometric structure of the Fatou set is described. We will show that each connected component of the Fatou set is an infinitely connected domain. Figure \ref{fatoucomp} is an illustration diagram of the Fatou set  (or Figure 1 in \cite{Bishop2018}). Recall that $f$ is univalent on each component of $R_j\cdot\Om^p_{n_j}$

For any positive integer $k$, let $\Om_k$ be a connected component of the Fatou set that contains the inner boundary component of $A_k$. Let $\ga_k$ be an outer boundary curve of $\Om_k$ satisfying that $\ga_k\subset V_k$ (by the definition of $V_k$ in \eqref{vkregion8-4-5} and Lemma \ref{valueest-17}), $\ga_k$ is a $C^1$ closed Jordan curve approximating some circle (by the discussions in the proof of Lemma \ref{packone}), $\ga_k$ separates $\Om_k$ from $\infty$, and $\ga_k$ is also an inner boundary curve of $\Om_{k+1}$. Further, the inner boundary curve of $\Om_k$ is also the outer boundary $\ga_{k-1}$ of $\Om_{k-1}$ for $k\geq2$.

Next, except for the inner and outer boundary curves of $\Om_k$, we classify other boundary components of $\Om_k$, which are also curves. The other boundary components of $\Om_k$ will be put into different ``levels" according to the following rule: components of level $j\geq k$ of $\Om_k$ are those curves which are mapped onto the curves $\ga_j$ by $f^{j-k+1}$.

These components are determined by the petals in each $R_j\cdot \Om^p_{n_j}$, where the set $W^n_k$ introduced in Lemma \ref{smalldim-1} will be used. This will imply that there are infinitely many holes in $\Om_k$, or, $\Om_k$ is infinitely connected domain.

First, we study the case $j=k$, that is, $f^{-1}(\ga_k)\cap (R_k\cdot \Om^p_{n_k})$. Since there are $n_k$ connected components in the $R_k\cdot\Om^p_{n_k}$, and these components are in $W^1_k$ , where $W^1_k=A_k\cap f^{-1}(A_k)$ is specified in Lemma \ref{smalldim-1} and Lemma \ref{inequ-15} is used.

Second, we consider the case $j=k+1$, $f^{-2}(\ga_{k+1})\cap  (R_k\cdot \Om^p_{n_k})$, in other words, the orbits travel from the petal $R_k\cdot \Om^p_{n_k}$ to the petal  $R_{k+1}\cdot \Om^p_{n_{k+1}}$, then $\ga_{k+1}$. So, there are  $n_k\cdot n_{k+1}$ connected components by Remark \ref{geofact-8-6-2}. This is corresponding to $W^{2}_{k+1}$, where $W^{2}_{k+1}$ is some component of $f^{-2}(A_{k+1})$ that lie inside $A_k$, and is introduced in the proof of Lemma \ref{smalldim-1}.

Inductively, we can study the case  $j>k+1$, that is, the orbits go through $j-k+1$ petals, $R_k\cdot\Om^{p}_{n_k}$, $R_{k+1}\cdot\Om^{p}_{n_{k+1}}$,...,$R_{j}\cdot\Om^p_{n_j}$, the total number of connected components is $n_k\times n_{k+1}\times\cdots\times n_j$ by Remark \ref{geofact-8-6-2}.

Now, we consider the critical points in $\Om_k$. By the calculation of critical points in \eqref{deriva-1}--\eqref{deriva-3}, the solutions of  $z^{s_k-1}=0$ do not give any critical point. The critical points are solutions of \eqref{deriva-3}, the total number is $n_k$. The map is a $m_k$-to-$1$ branched cover from $\Om_k$ to $\Om_{k+1}$, with the outer boundary mapping to the outer boundary (as a $m_k$-to-$1$ map), the inner boundary is mapped to the inner boundary (as a $m_{k-1}$-to-$1$ map), where $\Om_k$ is open.

For $k\leq0$, the Fatou components $\Om_k$ are defined as inverse images of $\Om_{k+1}$ under $f$.
By applying similar discussions as in \eqref{degreeest-8-4-3}, one has
 \begin{itemize}
 \item  $f$ is a $2^N$-to-$1$ covering map for $-T\leq k\leq0$;
   \item  $f$ is a $1$-to-$1$ conformal for $k<-T$.
\end{itemize}
By Remark \ref{equ2021-7-29-1} and Lemma \ref{criticalfatou-8-17-3}, all the critical points of $f$ are in the components of $\Om_k$ for $k=-T$ and $k\geq1$, every other component of the Fatou set is a conformal image of one of these and hence has the same geometry as $\Om_k$ for some $k\geq-T$, up to bounded distortion.

\subsection{Packing dimension}\label{packingdim-1}

In this subsection, we show the packing dimension is equal to $1$, where the
packing dimension agrees with the local upper Minkowski dimension for Julia sets of entire functions \cite{RipponStallard2005}.

\begin{lemma}\cite[Lemma 20.1]{Bishop2018}\label{equest-5}
Suppose $\Om$ is a bounded open set containing open subsets $\{\Om_j\}$ such that the measure of $\Om\setminus\cup\Om_j$ is zero, then for any $1\leq s\leq 2$, one has
\beqq
\sum_{Q\in \mathcal{W}(\Om)}\mbox{diam}(Q)^s\leq\sum_{j}\sum_{Q\in\mathcal{W}(\Om_j)}\mbox{diam}(Q)^s,
\eeqq
where $\mathcal{W}(\Om)$ and $\mathcal{W}(\Om_j)$ are Whitney decomposition of $\Om$ and $\Om_j$, respectively.
\end{lemma}

\begin{lemma}\cite[Lemma 20.2]{Bishop2018}\label{equest-6}
If $f:\Om_1\to\Om_2$ is bi-Lipschitz, then for any $0<s\leq2$, one has
\beqq
\sum_{Q\in\mathcal{W}(\Om_1)}\mbox{diam}(Q)^s\simeq\sum_{Q^{\prime}\in\mathcal{W}(\Om_2)}\mbox{diam}(Q^{\prime})^s.
\eeqq
\end{lemma}

\begin{lemma}\cite[Theorem 20.3]{Bishop2018}\label{equest-7}
For any annulus $\{z\in\mathbb{C}:\ r\leq|z|\leq r+\de\}$, the $t$-Whitney sum is
\beqq
O\bigg(\frac{1}{t}\cdot\de^{t-1}\cdot r^t\bigg).
\eeqq
\end{lemma}

\begin{theorem}
For $f$ satisfying the above hypothesis, one has $\mbox{Pdim}(\mathcal{J}(f))=1$.
\end{theorem}

\begin{proof}
The packing dimension is given by the upper Minkowski dimension of bounded pieces of the Julia set by Lemma \ref{equest-2}. The upper Minkowski dimension can be estimated by Lemma \ref{equest-3} and the Whitney decomposition.

By applying similar arguments used in the proof of Theorem 20.3 in \cite{Bishop2018},  Lemmas \ref{equest-5}---\ref{equest-7}, and the shape of the Fatou set, we could obtain this conclusion.
\end{proof}

\section*{Acknowledgments}
The authors appreciate Prof. Christopher Bishop, whose continuous encouragement and support made it possible to finish this work. The authors also appreciate the anonymous
reviewer for pointing out some mistakes on the application of the Koebe distortion theorem in the conformal estimates and valuable comments.

\end{document}